\documentclass[a4paper,11pt]{amsart}
  \usepackage{amsmath}
  \usepackage{url,graphicx}
  \usepackage{amssymb, color, pstricks-add, manfnt}
  \usepackage{amsmath, amsthm,  amsfonts}

  \theoremstyle{plain}
  \newtheorem{Theorem}{Theorem}[section]
  \newtheorem{Lemma}{Lemma}[section]

  \newtheorem{Corollary}{Corollary}[section]

  \numberwithin{equation}{section}
  \numberwithin{figure}{section}

\renewcommand{\baselinestretch}{1.00}
\begin{document}

\title{On the Neumann problem for Monge-Amp\`ere type equations}

\author{Feida Jiang}
\address{Yau Mathematical Sciences Center, Tsinghua University, Beijing 100084, P.R. China}
\address{Centre for Mathematics and Its Applications, The Australian National University, Canberra ACT 0200, Australia}
\email{jiangfeida@math.tsinghua.edu.cn}

\author{Neil S. Trudinger}
\address{Centre for Mathematics and Its Applications, The Australian National University, Canberra ACT 0200, Australia}
\email{Neil.Trudinger@anu.edu.au}

\author{Ni Xiang}
\address{Faculty of Mathematics and Statistics, Hubei Key Laboratory of Applied Mathematics, Hubei University, Wuhan 430062, P.R. China}
\email{nixiang@hubu.edu.cn}

\thanks{This work was started when the first two authors met at the Mathematical Sciences Center, Tsinghua University, in May 2013. It was supported by the National Natural Science Foundation of China(No.11401306, No.11101132), the Australian Research Council(No.DP1094303), China Postdoctoral Science Foundation(No.2015M571010), the Jiangsu Natural Science Foundation of China(No.BK20140126) and Foundation of Hubei Provincial Department of Education(No.Q20120105).}


\date{\today}


\abstract {In this paper, we study the global regularity for regular
Monge-Amp\`ere type equations associated with a semilinear Neumann boundary conditions.
By establishing {\it a priori} estimates for second order derivatives, the
classical solvability of the Neumann boundary value problem is proved under natural conditions.
The techniques build upon the delicate and intricate treatment of the standard Monge-Amp\`ere case
by Lions, Trudinger and Urbas in 1986 and the recent barrier constructions and second derivative bounds
by Jiang, Trudinger and Yang for the Dirichlet problem. We also consider more general oblique boundary
 value problems in the strictly regular case.}
\endabstract


\maketitle

\baselineskip=12.8pt
\parskip=3pt
\renewcommand{\baselinestretch}{1.38}

\section{Introduction}\label{Section 1}

\vskip10pt
In this paper, we consider the following semilinear Neumann boundary value problem for the Monge-Amp\`ere type equation
\begin{eqnarray}
&\det [D^2u-A(x,u,Du)]  =  B(x,u,Du),   & \quad \mbox{in } \ \Omega,\label{1.1}\\
&              D_\nu u  = \varphi(x,u), & \quad \mbox{on } \  \partial\Omega,\label{1.2}
\end{eqnarray}
where $\Omega$ is a bounded domain in $n$ dimensional Euclidean space $\mathbb{R}^n$ with smooth boundary, $Du$ and $D^2u$ denote the gradient vector and the Hessian matrix of the second order derivatives of the function $u:\Omega\rightarrow \mathbb{R}$ respectively, $A$ is a given $n\times n$ symmetric matrix function defined on $\Omega \times \mathbb{R} \times \mathbb{R}^n$, $B$ is a positive scalar valued function on $\Omega \times \mathbb{R} \times \mathbb{R}^n$, $\varphi$ is a scalar valued function defined on $\partial \Omega \times \mathbb{R}$ and $\nu$ is the unit inner normal vector field on $\partial \Omega$. As usual, we use $x$, $z$, $p$, $r$ to denote points in $\Omega$, $\mathbb{R}$, $\mathbb{R}^n$, $\mathbb{R}^{n\times n}$ respectively.
A solution $u\in C^2(\Omega)$ of equation \eqref{1.1} is
elliptic  when the augmented Hessian matrix $Mu =  D^2u-A(x,u,Du) $ is positive definite,
that is $Mu>0$, which implies $B>0$. Also, a function $u$ satisfying $Mu>0$ is called an elliptic function of the equation \eqref{1.1}. Since the matrix $A$ determines the  augmented Hessian matrix $Mu$, we also call an elliptic solution (or function) an $A$-admissible solution (or function) or, by analogy with the case $A=0$, an $A$-convex solution (or function).

We shall establish an existence theorem together with {\it a priori} estimates for elliptic solutions of the Neumann boundary value problem \eqref{1.1}-\eqref{1.2} in this paper, which extend the special case where $A$ is independent of $p$ in \cite{LTU1986}. For this purpose, we need appropriate  assumptions on $A$, $B$, $\varphi$ and $\Omega$. Assume that the matrix $A$ is twice differentiable with respect to $p$ and $A$, $B$ and $\varphi$ are  differentiable with respect to $z$. Following \cite{Tru2006}, we call the matrix $A$ regular in $\Omega$ if $A$ is co-dimension one convex with respect to $p$, in the sense that
\begin{equation}\label{regular}
A_{ij,kl}(x,z,p)\xi_i\xi_j\eta_k\eta_l \geq 0,
\end{equation}
for all $(x,z,p)\in\Omega\times \mathbb{R}\times\mathbb{R}^n$,
$\xi,\eta\in\mathbb{R}^n$, $\xi\perp\eta$,
where $A_{ij,kl}=D^2_{p_kp_l}A_{ij}$.
If the inequality  \eqref{regular} is strict, then the matrix $A$ is called strictly regular.
We also define the matrix  $A$ to be non-decreasing,  (strictly increasing), with respect to $z$, if
\begin{equation}\label{monotone A}
D_zA_{ij}(x,z,p)\xi_i\xi_j\ge 0, (>0),
\end{equation}
for all $(x,z,p)\in\Omega\times \mathbb{R}\times\mathbb{R}^n$, $\xi\in \mathbb{R}^n$.
The inhomogeneous term $B$ and boundary function $\varphi$ are also non-decreasing, (strictly increasing),
with respect to $z$, if
\begin{equation}\label{monotone B}
B_z(x,z,p)\ge 0, (>0),
\end{equation}
for all $(x,z,p)\in\Omega\times \mathbb{R}\times\mathbb{R}^n$ and
\begin{equation}\label{monotone varphi}
\varphi_z(x,z)\ge 0, (>0),
\end{equation}
for all $(x,z)\in\partial\Omega\times \mathbb{R}$.
Note that if we write the boundary value problem \eqref{1.1}-\eqref{1.2} in the general form
\begin{eqnarray}
&\mathcal F[u]:=F(x,u,Du,D^2u) = 0,   & \quad \mbox{in } \ \Omega,\label{1.7}\\
&   \mathcal G[u]:=G(x,u,Du)  = 0, & \quad \mbox{on } \  \partial\Omega,\label{1.8}
\end{eqnarray}
where $F$ and $G$ are defined by
\begin{eqnarray}
&F(x,z,p,r) =\det [r-A(x,z,p)] - B(x,z,p), \label{1.9}\\
&G(x,z,p) = \nu \cdot p -\varphi(x,z), \label{1.10}
\end{eqnarray}
then $A,B$ and $\varphi$ non-decreasing, (strictly increasing), in $z$, correspond to the standard
monotonicity conditions, $F_z\le 0, G_z\le 0$, ($F_z< 0, G_z< 0$) for symmetric matrices $r$ satisfying
$r > A(x,z,p)$, that is for points $(x,z,p,r) \in \Omega \times \mathbb{R} \times \mathbb{R}^n\times
\mathbb{R}^{n\times n}$, where $\mathcal F$ is elliptic.

As with \cite{LTU1986}, we also need the domain $\Omega$ to satisfy an appropriate uniform convexity condition.  Adapting \cite {Tru2006}, we define the domain $\Omega$ to be  uniformly $A$-convex, ($A$-convex),
with respect to the boundary function $\varphi$ and an interval valued function $\mathcal I$ on $\partial\Omega$
if $\Omega\in C^2$ and
 \begin{equation}\label{A convexity}
  (D_i\nu_j(x) -D_{p_k}A_{ij}(x,z,p)\nu_k)\tau_i\tau_j < 0, (\le 0),
  \end{equation}
  for all $(x,z,p)\in\partial\Omega\times \mathbb{R}\times\mathbb{R}^n$, satisfying
  $p\cdot\nu(x) \ge \varphi( x,z)$, $z\in\mathcal I(x)$ and vectors
  $\tau=\tau(x)$ tangent to $\partial\Omega$. For a given function $u_0$ on $\partial\Omega$, we define $\Omega$ to be  uniformly $A$-convex, ($A$-convex),
with respect to  $\varphi$ and $u_0$ if \eqref{A convexity} holds for all $p\cdot\nu(x) \ge \varphi( x,u_0(x))$, that is $\mathcal I = \{u_0\}$.

From the regularity of $A$ \eqref{regular},  we can equivalently replace the boundary inequality $p\cdot\nu \ge \varphi( x,z)$ by the boundary equality $p\cdot\nu = \varphi( x,z)$, in the above definitions, as  $D_{p_\nu}A_{ij}(x,z,p)\tau_i\tau_j$ is then non-decreasing with respect to $p_\nu$. This leads us to a further definition which is independent of the boundary condition \eqref{1.2}. Namely $\Omega$ is uniformly $A$-convex with respect to
$u\in C^1(\bar\Omega)$ if
\begin{equation}\label{A convexity wrt u}
  (D_i\nu_j -D_{p_k}A_{ij}(\cdot,u,Du)\nu_k)\tau_i\tau_j \le -\delta_0, \quad \mbox{on } \  \partial\Omega,
  \end{equation} for all vectors
  $\tau=\tau(x)$ tangent to $\partial\Omega$. Accordingly if $A$ is regular,  $\Omega$ is  uniformly $A$-convex
with respect to  $\varphi$ and $u$ and $u$ satisfies \eqref{1.2}, it follows that $\Omega$ is uniformly $A$-convex with respect to $u$.

In order to use the regularity of $A$ in its most general form, we will need  to assume the existence of a
 supersolution $\bar u$ to  \eqref{1.1} satisfying
\begin{equation}
\det [D^2\bar u-A(x,\bar u,D\bar u)]  \le  B(x,\bar u,D\bar u),    \quad \mbox{in } \ \Omega,\label{super equation}\\
\end{equation}
together with the same boundary condition,
\begin{equation}
D_\nu \bar u =  \varphi(x,\bar u),  \quad \mbox{on } \  \partial\Omega. \label{super boundary}
\end{equation}
We then have the following global second derivative estimate.

\begin{Theorem}\label{Th1.1}
Let $u\in C^4(\Omega)\cap C^{3}(\bar \Omega)$ be an elliptic solution of the Neumann problem \eqref{1.1}-\eqref{1.2} in a $C^{3,1}$  domain $\Omega\subset\mathbb{R}^n$, which is uniformly $A$-convex with respect to $u$, where $A\in C^2(\bar \Omega\times \mathbb{R}\times \mathbb{R}^n)$ is regular and non-decreasing, $B>0,\in C^2(\bar \Omega\times \mathbb{R}\times \mathbb{R}^n)$ is non-decreasing and $\varphi\in C^{2,1}(\partial \Omega\times \mathbb{R})$ is non-decreasing. Suppose there exists an elliptic supersolution $\bar u\in C^2(\bar \Omega)$ satisfying \eqref{super equation}-\eqref{super boundary}.
Then we have the estimate
\begin{equation}\label{global C2 bound}
\sup\limits_\Omega |D^2 u|\le C,
\end{equation}
where $C$ is a constant depending on $n, A, B, \Omega, \bar u, \varphi, \delta_0$, and $|u|_{1;\Omega}$.
\end{Theorem}

Theorem  \ref{Th1.1} extends Theorem 3.3 in \cite{LTU1986} except for the supersolution hypothesis as a supersolution is constructed in \cite{LTU1986} in the course of the proof. We also point out that, as in
\cite{LTU1986}, the restriction to the Neumann condition is critical for our proof and moreover as shown by the Pogorelov example, (see  \cite{Urbas1995}, \cite{Wang1992}),  one cannot generally expect second derivative estimates and classical solutions of \eqref{1.1}-\eqref{1.2} for $A=0$, when the geometric normal $\nu$ is replaced by an oblique vector $\beta$ satisfying $\beta\cdot\nu > 0$, that is in \eqref{1.10},
\begin{equation}\label{G in oblique form}
G(x,z,p) = \beta\cdot p - \varphi(x,z),
\end{equation}
no matter how smooth $\beta$, $\varphi$, $B$ and $\partial \Omega$ are. However if the matrix function $A$ is
strictly regular on $\bar \Omega$ so that we have a positive lower bound in \eqref{regular} when $z$ and $p$ are bounded, then the proof is much simpler and also embraces oblique boundary conditions. Moreover in this case the monotonicity and supersolution hypotheses in Theorem  \ref{Th1.1} can be dispensed with. Typically second derivative behaviour for equation \eqref{1.1} in the strictly regular case is closer to that for uniformly elliptic equations while the challenge in the general case is to carry over the more intricate Monge-Amp\`ere case,  $A = 0$.
Following \cite{LTU1986}, we can also relax the supersolution hypothesis for uniformly convex domains in the special case when $D_{px} A = 0$, that is
\begin{equation} \label {special case}
A(x,z,p) = A_0(x,z) + A_1(z,p),
\end{equation}
where $A_0\in C^2(\bar \Omega\times \mathbb{R})$ and $A_1\in C^2(\mathbb{R}\times \mathbb{R}^n)$ is regular.

From Theorem \ref{Th1.1}, we obtain classical existence theorems for \eqref{1.1}-\eqref{1.2} under further hypotheses ensuring estimates for solutions and their gradients. For solution estimates, by virtue of the comparison principle we can simply assume the existence of  bounded subsolutions and supersolutions.

However more specific conditions for solution bounds will be treated in Section \ref{Section 3} of this paper, including an extension of the Bakel'man condition in Theorem 2.1 of \cite{LTU1986}. For the gradient estimate we adopt the same structure condition used for the Dirichlet problem in \cite{JTY2013}, namely
\begin{equation}\label{QS}
A(x,z,p) \ge - \mu_0 (1+ |p|^2)I,
\end{equation}
for all $x\in \Omega$, $|z|\le M_0$, $p\in \mathbb{R}^n$ and some positive constant $\mu_0$ depending on the constant $M_0$. Condition \eqref{QS} provides a simple gradient bound for $A$-convex functions $u$ in terms of a lower bound for $D_\nu u$  on the boundary. Combining the second derivative bounds with the lower order bounds and the global second derivative H\"{o}lder estimates as in \cite{LieTru1986,LT1986,LTU1986,Tru1984}, we establish the following existence result by the method of continuity.

\begin{Theorem}\label{Th1.2}
Suppose that $A, B, \varphi, \bar u$ and $\Omega$ satisfy the hypotheses of Theorem \ref{Th1.1}, with either $A$, $B$
 or $\varphi$ being strictly increasing. Assume also  condition \eqref{QS} and
 that there exists an elliptic subsolution $\underline u \in
C^2(\Omega)\cap C^1(\bar \Omega)$ of equation \eqref{1.1}, with $D_\nu \underline u \ge \varphi(\cdot, \underline u)$
on $\partial\Omega$ and that $\Omega$ is uniformly $A$-convex with respect to $\varphi$ and
$\mathcal I = [\underline u, \bar u]$, in the sense of \eqref{A convexity}.
Then the Neumann boundary value problem
\eqref{1.1}-\eqref{1.2} has a unique elliptic solution $u\in
C^{3,\alpha}(\bar \Omega)$ for any $\alpha<1$.
\end{Theorem}

The uniqueness of solutions follows from the comparison principle for elliptic solutions of general oblique boundary value problems, \eqref{1.7}-\eqref{1.8}; see Lemma \ref{comparison}. The regularity for the solution $u$ in Theorem \ref{Th1.2} can be improved by the linear elliptic theory \cite{GTbook} if the data are sufficiently smooth. For example, if $A$, $B$, $\varphi$ and $\partial \Omega$ are $C^\infty$, then the solution $u\in C^\infty(\bar \Omega)$. From the monotonicity of $\varphi$, it is also enough to assume  \eqref{A convexity} only holds for
$p\cdot \nu \ge \varphi(\cdot, \underline u)$ and $\underline u \le z \le \bar u$. Moreover if $A$ is independent of $z$, there is no need for the last inequality. Also taking account of our remarks after the statement of Theorem \ref{Th1.1}, we only need to assume the supersolution $\bar u$ satisfies \eqref{super equation} at points where it is elliptic and the boundary inequality $D_\nu \bar u \le \varphi(\cdot, \bar u)$, instead of \eqref{super boundary}, if either
 $A$ satisfies \eqref {special case} with $\Omega$ also uniformly convex  or $A$ is strictly regular in $\bar\Omega$.

The regular condition was originally introduced in \cite {MTW2005} in its strict form for interior regularity of potential functions in optimal transportation with the weak form \eqref{regular} subsequently introduced in \cite{TruWang2009} for global regularity; (see also \cite{Tru2006}). It was subsequently shown to be sharp for $C^1$ regularity of potential functions in
\cite{Loeper2009}.
Optimal transportation equations are special cases of prescribed Jacobian equations, which have the general form,
 \begin{equation}\label{PJE}
|\det DY (\cdot,u,Du)|  =  \psi(\cdot,u,Du),
\end{equation}
where $Y$ is a $C^1$ mapping from $\Omega\times\mathbb{R}\times\mathbb{R}^n$ into $\mathbb{R}^n$, $\psi$ is a non-negative scalar valued function on $\Omega \times \mathbb{R} \times \mathbb{R}^n$.
Assuming $\det Y_p\ne 0$, we see that for elliptic solutions, equation \eqref{PJE} can be written in the form
\eqref{1.1} with
\begin{equation} \label{def:AB}
A = - Y^{-1}_p(Y_x+Y_z\otimes p),
\quad B = (\det Y_p)^{-1}\psi.
\end{equation}
 The natural boundary value problem for the prescribed Jacobian equation is the \emph{second boundary
value problem}  to prescribe the image,
\begin{equation} \label{bvp}
Tu(\Omega): = Y (\cdot,u,Du)(\Omega) = \Omega^*,
\end{equation}
where $\Omega^*$ is another given domain in  $\mathbb{R}^n$. The global regularity of the second boundary value problem \eqref{PJE}-\eqref{bvp} has already been studied in \cite{Caffarelli1996, Urbas1997, TruWang2009, Tru2008, Tru2013} for different forms of the mapping $Y$.  As shown   in \cite{TruWang2009} in the optimal transportation case and in \cite{Tru2008} in the general case, condition \eqref{bvp} implies an oblique nonlinear boundary condition for elliptic functions $u$, that is \eqref{1.8} holds for a function $G \in C^1(\partial\Omega\times\mathbb{R}\times\mathbb{R}^n)$ with
\begin{equation}\label{oblique}
G_p(\cdot,u,Du)\cdot\nu > 0, \quad \mbox{on } \  \partial\Omega.
\end{equation}
The crucial estimate in these papers is the control on the obliqueness, that is an estimate of the form, $G_p\cdot\nu\ge\delta$ for a positive constant  $\delta$ and this is done in \cite{TruWang2009} in the optimal transportation case, and extended to the general case in \cite{Tru2008}, under appropriate uniform convexity conditions on the domain and target, with the latter equivalent to the uniform concavity of the function $G$ with respect to the $p$ variables. Because we are defining obliqueness with respect to the inner normal, in agreement with \cite{LTU1986}, our function $G$ is the negative of that in
\cite{Urbas1997,Tru2008,TruWang2009}.  Once the obliqueness is estimated, the boundary second derivative bounds follow in \cite{TruWang2009, Tru2008, JT2014}   from the same uniform convexity conditions, together with the regular condition \eqref{regular}, similarly to  the Monge-Amp\`ere case in \cite{Urbas1998}. Note that the uniform concavity of $G$ excludes the Neumann condition treated here and moreover the derivation of the boundary $C^2$ estimate is  much simpler, being somewhat analogous to using the strict regular condition. We also point out a recent paper \cite{CLY} considering optimal transportation on a hemisphere  where the obliqueness is estimated without using any uniform convexity of domains, which still gives the boundary $C^2$ estimate in the two dimensional case. Prescribed Jacobian equations also arise
in  geometric optics where solutions correspond to reflectors or refractors transmitting light rays from a source to a target with prescribed intensities; (see for example \cite{Wang1996, KW2010, Tru2011, Tru2014, JT2014, LT2014} and references therein).

On the geometric side, the Neumann boundary value problem in the more general context of augmented Hessian equations on manifolds arises in the study of the higher order Yamabe problem in conformal geometry; (see
 \cite{SSChen2007,SSChen2009,JLL2007,LiLi2006,LiNguyen2014}). To explain this we let $(\mathcal{M},g)$ be a smooth compact Riemannian manifold of dimension $n\ge 3$ with nonempty smooth boundary $\partial \mathcal{M}$,  let
$A_g$ denote the Schouten tensor of the metric $g$ and let  $\lambda(A_g)=(\lambda_1(A_g),\cdots, \lambda_n(A_g))$ denote the eigenvalues of $A_g$. Let  $\Gamma\subset \mathbb{R}^n$ be an open convex symmetric cone with vertex at the origin and $f$ be a smooth symmetric function in $\Gamma$. The fully nonlinear Yamabe problem on manifolds with boundary is to find a metric $\tilde g$ in the conformal class of the metric $g$ with a prescribed function of eigenvalues of the Schouten tensor and prescribed mean curvature. For example, for a given constant $c\in \mathbb{R}$, we are interested in finding a metric $\tilde g$ conformal to $g$ such that
\begin{equation}\label{Fully nonlinear Yamabe with boundary}
\begin{array}{cl}
F(A_{\tilde g}):=f(\lambda(A_{\tilde g}))=1, &  \lambda(A_{\tilde g}) \in \Gamma \ {\rm on} \ \mathcal{M},\\
h_{\tilde g}=c,                              &  {\rm on}\ \partial \mathcal{M},
\end{array}
\end{equation}
where $h_{\tilde g}$ denotes the mean curvature of $\partial\mathcal{M}$ with respect to the inner normal. Writing $\tilde g = e^{-2u}g$ for some smooth function $u$ on $\mathcal{M}$, by the transformation laws for the Schouten tensor and mean curvature, the problem \eqref{Fully nonlinear Yamabe with boundary} is equivalent to the following semilinear Neumann boundary value problem
\begin{equation}\label{transformed Yamabe}
\begin{array}{cl}
\displaystyle f(\lambda_g(U))=e^{-2u}, &  \lambda_g(U) \in \Gamma \ {\rm on} \ \mathcal{M},\\
\displaystyle \frac{\partial u}{\partial \nu}=ce^{-u}-h_g,                   &  {\rm on}\ \partial \mathcal{M},
\end{array}
\end{equation}
with
$$U=\nabla^2 u + du\otimes du-\frac{1}{2}|\nabla u|^2g + A_g,$$
where $\lambda_g(U)$ denotes the eigenvalues of $U$ with respect to $g$, $\nu$ is the unit inner normal vector field to $\partial \mathcal{M}$, and $\nabla$ denotes the Levi-Civita connection with respect to $g$. If we choose $f=\det$ and $\Gamma=\Gamma_n:=\{\lambda=(\lambda_1,\cdots,\lambda_n)\in \mathbb{R}^n: \sum \lambda_i>0\}$, then we have an example \eqref{transformed Yamabe} of a semilinear Neumann boundary value problem \eqref{1.1}-\eqref{1.2} for a Monge-Amp\`ere type equation. In conclusion a prescribed mean curvature fully nonlinear Yamabe problem \eqref{Fully nonlinear Yamabe with boundary} is equivalent to a semilinear Neumann problem \eqref{transformed Yamabe} for an augmented Hessian equation.  The corresponding matrix functions in these cases will be strictly regular when expressed in terms of local coordinates so that in the Monge-Amp\`ere case strong local estimates are available, with second order estimates being considerably simpler than the general regular case we treat here. In fact, the particular Neumann boundary value problem \eqref{transformed Yamabe} with $f=\det$ has already been studied in \cite{JLL2007}. In the special case of Euclidean space $\mathbb{R}^n$, the matrix $A$ is given by
\begin{equation}\label{Euclidean Yamabe}
A=\frac{1}{2}|p|^2 I-p\otimes p,
\end{equation}
in which case our $A$-convexity condition \eqref{A convexity} reduces to simply $\kappa_1 + \varphi > 0, (\ge 0)$, 
where $\kappa_1$ denotes the minimum curvature of $\partial\Omega$.
The overall organisation of this paper follows that of the Dirichlet problem case \cite{JTY2013}, where again the main issue was to deal with the general case of regular $A$. Also here the strictly regular case is considerably simpler in the case of smooth data but in the optimal transportation case with only H\"{o}lder continuous densities local and global second derivative estimates were obtained in \cite{LTW2010, HJL2013}, in agreement with the uniformly elliptic case. In Section \ref{Section 2} we prove Theorem \ref{Th1.1}, which constitutes the heart of the paper. In Section \ref{Section 3} we provide the gradient estimate to complete the proof of Theorem \ref{Th1.2}, along with alternative solution bounds for more general oblique boundary value problems. In the optimal transportation case we also prove a Bakel'man type estimate for solutions which extends the Monge-Amp\`ere case in \cite{LTU1986}. In Section \ref{Section 4} we switch to the strictly regular case and prove first and second derivative bounds  for general oblique boundary value problems \eqref{1.8}, where $G$ is concave with respect to the $p$ variables, which extend the semilinear conditions \eqref{G in oblique form}. For this
purpose we extend our definition of $A$-convexity so that a $C^2$ domain $\Omega$ is  uniformly $A$-convex, ($A$-convex),
with respect to $G$ and an interval $\mathcal I$
if \eqref{A convexity} holds for all $(x,z,p)\in\partial\Omega\times \mathbb{R}\times\mathbb{R}^n$, satisfying
 $G(x,z,p) \ge 0$, $z\in\mathcal I$ and vectors $\tau$ tangent to $\partial\Omega$. When $G$ is independent of $z$, this corresponds to the $c$-convexity conditions from optimal transportation \cite {Tru2006,TruWang2009} and more  generally to the $Y$-convexity conditions for  prescribed Jacobian equations in \cite{Tru2008}. Finally we remark that
 a general theory of oblique boundary value problems for augmented Hessian equations, which embraces our results in Section 4, is presented in \cite{JT2015}.

\section{Second derivative estimates}\label{Section 2}
\vskip10pt

In this section, we shall derive the second order derivative estimates and complete the proof of Theorem \ref{Th1.1} by taking full advantage of the assumed $C^2$ supersolution $\bar u$. Note that we only need to get an upper bound for the second derivatives, since the lower bound can be derived from the ellipticity condition $D^2u-A>0$.

For the arguments below, we assume the functions
$\varphi$, $\nu$ can be smoothly extended to $\bar \Omega\times \mathbb{R}$ and $\bar \Omega$ respectively. We also assume that near the boundary, $\nu$ is extended to be constant in the normal directions. From the equation \eqref{1.1}, we have
\begin{equation}\label{3.3}
\tilde F[u]:=\log \det [D^2u-A(\cdot,u,Du)]=\tilde B(\cdot,u,Du),
\end{equation}
where $\tilde B \triangleq \log B$. We have $\frac{\partial \tilde F}{\partial
w_{ij}}=w^{ij}$, $\frac{\partial^2 \tilde F}{\partial
w_{ij}\partial w_{kl}}=-w^{ik}w^{jl}$, where
$\{w_{ij}\}\triangleq\{u_{ij}-A_{ij}\}$ denotes the augmented
Hessian matrix, and $\{w^{ij}\}$ denotes the inverse of the matrix
$\{w_{ij}\}$. We now introduce the following linearized operators of $\tilde F$ and \eqref{3.3},
\begin{equation}\label{3.4}
L\triangleq w^{ij}(D_{ij}-D_{p_l}A_{ij}(\cdot,u,Du)D_l), \quad \mathcal{L}v\triangleq Lv-D_{p_l}\tilde B(\cdot,u,Du)D_l.
\end{equation}
For convenience in later discussion, we denote $D_{\xi\eta}u\triangleq
D_{ij}u\xi_i\eta_j$, $w_{\xi\eta}\triangleq w_{ij}\xi_i\eta_j =
D_{ij}u\xi_i\eta_j-A_{ij}\xi_i\eta_j$ for any vectors $\xi$ and
$\eta$. As usual, $C$ denotes a constant depending on the known data
and may change from line to line in the context.

Before we start to deal with the second derivative estimates, we recall a fundamental lemma in \cite{JT2014, JTY2013}, which is also crucial to construct the global barrier function using the supersolution in our situation. We shall omit its proof, which is similar to those in \cite{JT2014, JTY2013}.

\begin{Lemma}\label{Lemma barrier}
Let $u\in C^2(\bar \Omega)$ be an elliptic solution of
\eqref{1.1}, $\tilde u \in C^2(\bar \Omega)$ be an elliptic
function of equation \eqref{1.1} in $\bar\Omega$ with $\tilde u \ge u$ in $\bar \Omega$, where $A$ is
regular and non-decreasing. Then
\begin{equation}\label{key inequality}
\mathcal L \left({e^{K(\tilde u-u)}}\right) \ge
\epsilon_1\sum\limits_{i} w^{ii}-C,
\end{equation}
holds in $\Omega$ for sufficiently large positive constant $K$ and
uniform positive constants $\epsilon_1, C$ depending on $A, B, \Omega, |u|_{1;\Omega}$ and $\tilde u$.
\end{Lemma}

We assume that the domain $\Omega$ is uniformly $A$-convex, with respect to $\varphi$ and $u$, and first consider the second derivative estimates on the boundary $\partial \Omega$ in nontangential  directions.
We introduce the tangential gradient operator $\delta=(\delta_1,\cdots, \delta_n)$, where $\delta_i=(\delta_{ij}-\nu_i\nu_j)D_j$. Applying this tangential operator to the boundary condition \eqref{1.2}, we have
\begin{equation}\label{3.1}
D_ku\delta_i \nu_k + \nu_k \delta_i D_ku=\delta_i\varphi,\quad {\rm on} \  \partial \Omega,
\end{equation}
hence we have
\begin{equation}\label{2.3}
|D_{\tau\nu}u|\leq C,\quad {\rm on} \  \partial \Omega,
\end{equation}
for any tangential vector field $\tau$.

We next deduce the estimate for $D_{\nu\nu}u$ on $\partial \Omega$. By a direct calculation, we have
\begin{equation}\label{3.5}
\begin{array}{ll}
Lu \!\!\!& =w^{ij}(D_{ij}u-D_{p_l}A_{ij}(\cdot,u,Du)D_lu)\\
   \!\!\!& =n-w^{ij}(A_{ij}-D_{p_l}A_{ij}(\cdot,u,Du)D_lu).
\end{array}
\end{equation}
Differentiating the equation \eqref{3.3} with respect to $x_k$, we have, for $k=1, \cdots, n$,
\begin{equation}
    \begin{array}{ll}
         \!&\!\! w^{ij}(D_{ij}u_k-D_kA_{ij}-D_zA_{ij}u_k-D_{p_l}A_{ij}D_{l}u_k)\\
        =\!&\!\! D_{x_k}\tilde B+ D_z\tilde B u_k + D_{p_l}\tilde B D_lu_k,
    \end{array}
\end{equation}
which implies
\begin{equation}\label{3.7}
Lu_k=D_{x_k}\tilde B +D_z\tilde B u_k + D_{p_l}\tilde B D_{kl}u+w^{ij}(D_kA_{ij}+D_zA_{ij}u_k), \quad{\rm for}\  k=1,\cdots,n.
\end{equation}
If we consider the function $h=\nu_kD_ku-\varphi(x,u)$, by \eqref{3.5} and \eqref{3.7}, we immediately have
\begin{equation}\label{3.8}
|Lh|\le C(1+\sum w^{ii}+|D^2u|),\quad \ {\rm in} \ \Omega.
\end{equation}
From the positivity of $B$ we can estimate,
\begin{equation}\label{2.10}
1\le C w^{ii}, \quad (w_{ii})^\frac{1}{n-1}\le Cw^{ii}.
\end{equation}
 Thus we obtain from
\eqref{3.8} and the boundary condition \eqref{1.2},
\begin{equation}\label{2.11}
|Lh|\le C(1+|D^2u|^\frac{n-2}{n-1}) \sum w^{ii} \  \ {\rm in} \ \Omega, \quad {\rm and}\quad h=0 \ \ {\rm on}\ \partial \Omega.
\end{equation}

From the uniform $A$-convexity of $\Omega$ \eqref{A convexity wrt u} and the regularity of $A$ , there exists a defining function, $\phi\in C^2(\bar \Omega)$,
satisfying $\phi=0$ on $\partial \Omega$, $D\phi \neq 0$ on
$\partial \Omega$ and $\phi<0$ in $\Omega$, together with the
inequality
\begin{equation}\label{domain A convexity}
D_{ij}\phi - D_{p_k}A_{ij}(\cdot,u,Du)D_k\phi \ge \delta_1 I,
\end{equation}
in a neighbourhood $\mathcal{N}$ of $\partial \Omega$, whenever $D_\nu u \ge \varphi(x,u)$, where $\delta_1$ is a positive constant  and $I$
denotes the identity matrix, with $\mathcal{N}$ and $\delta_1$ depending also on $\delta_0$, $A$ and
$|u|_{1;\Omega}$.
We remark that \eqref{domain A convexity} follows from \eqref{A convexity wrt u},
using the continuity of $D_pA$ with respect to $x$ and $z$  together with an appropriate extension of the distance function, as in for example \cite{GTbook,TruWang2009}. In particular, we can take $\phi=-d+td^2$ near $\partial\Omega$, for a large enough positive constant $t$, where $d(x)={\rm{dist}} (x,\partial\Omega)$ is the distance function of $\Omega$. Accordingly
\begin{equation}\label{2.13}
L\phi \ge  \delta_1 \sum w^{ii},
\end{equation}
for  $h\ge 0$, $d<d_0$, for a positive constant $d_0$  also depending  on $\delta_0$, $A$ and $|u|_{1;\Omega}$.  By \eqref{2.11}, \eqref{2.13} and choosing $- \phi$ as a barrier function, a standard barrier argument leads to
$$D_\nu h\le C(1 + M_2^\frac{n-2}{n-1}), \quad {\rm on}\ \partial \Omega,$$
where $M_2 = \sup\limits_\Omega |D^2 u|$,
so that we have the estimate
\begin{equation}\label{3.10}
D_{\nu\nu} u\le C(1 + M_2)^\frac{n-2}{n-1}, \quad {\rm on}\ \partial \Omega.
\end{equation}
 We conclude from \eqref{2.3}, \eqref{3.10} and the ellipticity of $u$ that
\begin{equation}\label{2.15}
|D_{\nu\xi} u|\le  C(1 + M_2)^\frac{n-2}{n-1}, \quad {\rm on}\ \partial \Omega,
\end{equation}
for any direction $\xi$. We remark that if $B$ is independent of $p$ or $n=2$, then the term in $M_2$ is not present in \eqref{2.15}.

We have now established the mixed tangential normal derivative bound and the double normal derivative bound on $\partial \Omega$ so that it  remains to bound the double tangential second derivatives on $\partial \Omega$. We shall adapt
the  delicate method in \cite{LTU1986}, which is specific for the Neumann boundary value problem, to obtain the double tangential derivative bound on the boundary and consequently the global second derivative bound.

\begin{proof}[Proof of Theorem \ref{Th1.1}]
First we note from the comparison principle, Lemma 3.1, that $\bar u \ge u$ in $\Omega$ or $u-\bar u$ is a constant.
Discarding the second case, we  modify the elliptic supersolution $\bar u$ by adding a perturbation function $-a \phi$ , where $a$ is a small positive constant and $\phi$ is the defining function of the domain $\Omega$ satisfying $\phi=0$ on $\partial \Omega$, $\phi <0$ in $\Omega$ and $D_\nu \phi =-1$ on $\partial \Omega$. Note that the new function $\tilde u = \bar u - a \phi$ is still uniformly elliptic in $\Omega$ if $a$ is sufficiently small. Also, by a direct computation, we have
\begin{equation}
\begin{array}{rl}
D_\nu (\tilde u - u) = \!\!&\!\! D_{\nu}\bar u - D_\nu u - aD_\nu \phi\\
                     = \!\!&\!\! \varphi(\cdot,\bar u) - \varphi(\cdot,u) + a\\
                   \ge \!\!&\!\! a,
\end{array}
\end{equation}
on $\partial \Omega$, where the non-decreasing of $\varphi$ and $\bar u \ge u$ on $\partial \Omega$ are used. If we define a function with the form $\Phi= e^{K(\tilde u-u)}$ with a positive constant $K$, we then have $D_\nu \Phi \ge Ka > 0$ on $\partial \Omega$.
We now introduce an auxiliary function $v$, given by
\begin{equation}\label{aux fn}
v=v(\cdot, \xi)= e^{\frac{\alpha}{2} |Du|^2+ \kappa \Phi}(w_{\xi\xi}-v'(\cdot, \xi)),
\end{equation}
for $x\in \bar \Omega$, $|\xi| = 1$, where $\alpha$, $\kappa$ are positive constants to be determined,
$$\Phi=\frac{1}{\epsilon_1}e^{K(\tilde u-u)}$$
is the barrier function in Lemma \ref{Lemma barrier} with the above constructed $\tilde u$, and $v'$ is  defined by
\begin{equation}
v'(\cdot, \xi)= 2(\xi\cdot \nu)\xi'_i(D_i \varphi(\cdot,u) - D_kuD_i\nu_k - A_{ij}\nu_j),
\end{equation}
with $\xi'=\xi-(\xi\cdot \nu)\nu$. Here $\nu$ is a $C^{2,1}(\bar \Omega)$ extension of the inner unit normal vector field on $\partial \Omega$. The strategy of our proof is to estimate $v$ at a maximum point in $\bar\Omega$ and vector
$\xi$, in the same form as \eqref{2.15}. from this we conclude a corresponding global estimate for $D^2 u$ in
$\Omega$ from which follows the desired estimate \eqref{global C2 bound}.

\textbf{Case 1.} We suppose that $v$ takes its maximum at an interior point $x_0\in \Omega$ and a vector $\xi$. Let
\begin{equation}
H = \log v =\log (w_{\xi\xi}-v') + \frac{\alpha}{2} |Du|^2 + \kappa \Phi,
\end{equation}
then the function $H$ also attains its maximum at the point $x_0\in \Omega$ and the unit vector $\xi$.
The following analysis follows the method of Pogorelov type estimates in \cite{TruWang2009}, with some modification, adapted from  \cite{LTU1986}, to handle the additional term $v'$. Accordingly  we have, at the point $x_0$,
\begin{equation}\label{Maximum Point}
\begin{array}{rl}
0 =  D_i H   = \!\!&\!\! \displaystyle \frac{D_i(w_{\xi\xi}-v')}{w_{\xi\xi}-v'} + \alpha D_ku D_{ik}u + \kappa D_i \Phi, \quad\quad\quad {\rm for}\  i=1\cdots n,\\
0\ge D_{ij}H = \!\!&\!\! \displaystyle \frac{D_{ij}(w_{\xi\xi}-v')}{w_{\xi\xi}-v'} - \frac{D_i(w_{\xi\xi}-v')D_j(w_{\xi\xi}-v')}{(w_{\xi\xi}-v')^2}\\
               \!\!&\!\! \displaystyle + \alpha (D_{ik}uD_{jk}u + D_ku D_{ijk}u) + \kappa D_{ij}\Phi,
\end{array}
\end{equation}
and consequently, at $x_0$
\begin{equation}\label{LH le 0}
\begin{array}{rl}
0 \ge \mathcal LH = \!\!&\!\! \displaystyle \frac{1}{w_{\xi\xi}-v'}L(w_{\xi\xi}-v') - \frac{1}{(w_{\xi\xi}-v')^2}w^{ij}{D_i(w_{\xi\xi}-v')D_j(w_{\xi\xi}-v')}\\
           \!\!&\!\! \displaystyle + \alpha w^{ij}D_{ik}uD_{jk}u + \alpha D_ku\mathcal Lu_k + \kappa L \Phi.
\end{array}
\end{equation}
Next, we shall estimate each term on the right hand side of \eqref{LH le 0}. We start with some identities. By differentiation of the equation \eqref{3.3} in the direction $\xi$, we have in accordance with \eqref{3.7},
\begin{equation}\label{once differentiated}
    \begin{array}{ll}
         \!&\!\! w^{ij}(D_{ij}u_\xi-D_\xi A_{ij}-D_zA_{ij}u_\xi-D_{p_l}A_{ij}D_{l}u_\xi)\\
        =\!&\!\! D_\xi\tilde B+ D_z\tilde B u_\xi + D_{p_l}\tilde B D_lu_\xi,
    \end{array}
\end{equation}
and a further differentiation in the direction of $\xi$ yields,
\begin{equation}\label{twice differentiated}
\begin{array}{ll}
 \!&\!\!\displaystyle w^{ij}[D_{ij}u_{\xi\xi}-D_{\xi\xi}A_{ij}-(D_{zz}A_{ij})(u_\xi)^2-(D_{p_kp_l}A_{ij})D_ku_\xi D_lu_\xi\\
 \!&\!\!\displaystyle \quad -(D_zA_{ij})u_{\xi\xi}-(D_{p_k}A_{ij})D_ku_{\xi\xi}-2(D_{\xi z}A_{ij})u_\xi\\
 \!&\!\!\displaystyle \quad -2(D_{\xi p_k}A_{ij})D_ku_\xi - 2(D_{zp_k}A_{ij})(D_ku_\xi)u_\xi ]\\
=\!&\!\!\displaystyle w^{ik}w^{jl}D_{\xi}w_{ij}D_{\xi}w_{kl} + D_{\xi\xi}\tilde B + (D_{zz}\tilde B)(u_\xi)^2 +
(D_{p_kp_l}B)D_ku_\xi D_lu_\xi\\
 \!&\!\!\displaystyle \quad + 2(D_{\xi z}\tilde B)u_\xi + 2(D_{\xi p_k}\tilde B)D_ku_\xi + 2(D_{zp_k}\tilde B)(D_ku_\xi)u_\xi\\
 \!&\!\!\displaystyle \quad + (D_z \tilde B)u_{\xi\xi} + (D_{p_k}\tilde B)D_k u_{\xi\xi}.
\end{array}
\end{equation}
Using \eqref{twice differentiated} and the regular condition \eqref{regular}, (see (3.9) in \cite{TruWang2009}), we have
\begin{equation}\label{Luxixi}
\mathcal Lu_{\xi\xi} \ge w^{ik}w^{jl}D_{\xi}w_{ij}D_{\xi}w_{kl} - C[(1+w_{ii})\mathcal T + (w_{ii})^2],
\end{equation}
where we denote $\mathcal T = w^{ii}$ to avoid any confusion with the usual summation convention. When calculating $LA_{\xi\xi}$, there will occur third derivative terms of $u$, which are controlled using  \eqref{once differentiated}. We then obtain
\begin{equation}\label{LAxixi}
|\mathcal LA_{\xi\xi}| \le C[(1+w_{ii})\mathcal T+w_{ii}]
\end{equation}
and by a similar calculation, we have
\begin{equation}\label{Lv'}
|\mathcal Lv'| \le C[(1+w_{ii})\mathcal T+w_{ii}].
\end{equation}
Combining \eqref{Luxixi}, \eqref{LAxixi} and \eqref{Lv'}, we have
\begin{equation}\label{Lwxixiv'}
\mathcal L(w_{\xi\xi}-v') \ge w^{ik}w^{jl}D_{\xi}w_{ij}D_{\xi}w_{kl} - C[(1+w_{ii})\mathcal T + (w_{ii})^2].
\end{equation}
By  Cauchy's inequality, we have
\begin{equation}\label{Cauchy's inequality}
w^{ij}{D_i(w_{\xi\xi}-v')D_j(w_{\xi\xi}-v')} \le (1+\theta)w^{ij}D_iw_{\xi\xi}D_jw_{\xi\xi} + C(\theta)w^{ij}D_iv'D_jv'
\end{equation}
for any $\theta>0$, where $C(\theta)$ is a positive constant depending on $\theta$.

By \eqref{key inequality}, \eqref{3.7}, \eqref{Lwxixiv'} and \eqref{Cauchy's inequality}, we then obtain from \eqref{LH le 0}
\begin{equation}\label{LH le 0'}
\begin{array}{rl}
0 \ge \!\!&\!\! \displaystyle \frac{1}{w_{\xi\xi}-v'}w^{ik}w^{jl}D_{\xi}w_{ij}D_{\xi}w_{kl} - \frac{1+\theta}{(w_{\xi\xi}-v')^2}w^{ij}{D_iw_{\xi\xi}D_jw_{\xi\xi}}\\
      \!\!&\!\! \displaystyle + \alpha w_{ii} + \kappa \mathcal T - C\{\frac{1}{w_{\xi\xi}-v'}[(1+w_{ii})\mathcal T + (w_{ii})^2]+\alpha+\kappa \}\\
      \!\!&\!\! \displaystyle  - \frac{C(\theta)}{(w_{\xi\xi}-v')^2}w^{ij}D_iv'D_jv'.
\end{array}
\end{equation}

Without loss of generality, we assume that $\{w_{ij}\}$ is diagonal at $x_0$ with maximum eigenvalue $w_{11}$. We can always assume that $w_{11}>1$ and is as large as we want; otherwise we are done. We proceed first to estimate the third derivative terms in \eqref{LH le 0'}. From the inequality (3.48) in \cite{LTU1986}, we have
\begin{equation} \label{2.30}
w^{ik}w^{jl}D_\xi w_{ij}D_{\xi}w_{kl}-\frac{1}{w_{11}}w^{ij}{D_iw_{\xi\xi}D_jw_{\xi\xi}} \ge 0.
\end{equation}
Moreover since $v^\prime$ is bounded, $w_{11}$ and $w_{\xi\xi}$ are comparable in the sense that for any
 $\theta > 0$, there exists a further constant $C(\theta)$ such that
 \begin{equation}\label{2.31}
|w_{11} - w_{\xi\xi} + v^\prime| < \theta w_{11},
 \end{equation}
 if $w_{11} > C(\theta)$.  From  \eqref{2.30} and \eqref{2.31}, we have
 \begin{equation}\label{2.32}
w^{ik}w^{jl}D_{\xi}w_{ij}D_{\xi}w_{kl} \ge \frac{1-\theta}{w_{\xi\xi}-v'}w^{ij}{D_iw_{\xi\xi}D_jw_{\xi\xi}}.
 \end{equation}
 Next we use  $D_iH=0$ in \eqref{Maximum Point}, to estimate
 \begin{equation}\label{2.33}
 \begin{array}{rl}
w^{ij}D_iw_{\xi\xi}D_jw_{\xi\xi} \le \!\!&\!\!   2w^{ii}[|D_iv^\prime|^2 +  (w_{\xi\xi}-v')^2 (\alpha D_kuD_{ik}u + \kappa D_i\Phi)^2] \\
                                 \le \!\!&\!\!   C[w^{ii} +  (w_{\xi\xi}-v')^2 (\alpha^2 w_{ii} + \kappa^2 \mathcal T)].
 \end{array}
 \end{equation}
Using \eqref{2.32} and \eqref{2.33} in \eqref{LH le 0'}, together with \eqref{2.31}, we then obtain for $w_{11} \ge C(\theta)$,
\begin{equation}\label{2.34}
  \alpha w_{ii} + \kappa \mathcal T \le C[ \alpha +\kappa + (1+\alpha^2\theta)w_{ii} + (1+\kappa^2 \theta) \mathcal T].
\end{equation}
  By choosing  $\alpha$, $\kappa$ large, and then fixing a small positive $\theta$, we thus obtain an estimate  $w_{ii}(x_0) \le C$, which implies a corresponding estimate for $|D^2u(x_0)|$.

\vspace{2mm}

\textbf{Case 2.} We consider the case $x_0\in \partial\Omega$, namely the function $v(x, \xi)= e^{\frac{\alpha}{2} |Du|^2+ \kappa \Phi}(w_{\xi\xi}-v')$ attains its maximum over $\bar \Omega$ at $x_0 \in \partial\Omega$ and a unit vector $\xi$. We then consider the following three subcases of different directions of $\xi$. For this we employ the key trick from \cite{LTU1986}.

\vspace{2mm}

{\it Subcase (i).} $\xi=\nu$, where $\nu$ is normal to $\partial\Omega$ at $x_0$. Since from \eqref{3.10} we already obtained the double normal derivative bound, we have
\begin{equation}\label{2.35}
v(x_0,\nu) \le C(1 + M_2)^\frac{n-2}{n-1}, \quad {\rm on}\ \partial \Omega.
\end{equation}

\vspace{2mm}

{\it Subcase (ii).} $\xi$ is neither normal nor tangential to $\partial\Omega$. The unit vector $\xi$ can be written as
\begin{equation}
\xi = (\xi\cdot\tau)\tau + (\xi\cdot\nu)\nu,
\end{equation}
where $\tau\in S^{n-1}$, with $\tau\cdot \nu=0$, $(\xi\cdot\tau)^2+(\xi\cdot\nu)^2=1$ and $\xi\cdot \nu\neq 0$. By the construction of $v'$, we have at $x_0$,
\begin{equation}
\begin{array}{rl}
w_{\xi\xi}= \!\!&\!\! (\xi\cdot\tau)^2w_{\tau\tau} + (\xi\cdot\nu)^2 w_{\nu\nu} + 2(\xi\cdot\tau)(\xi\cdot\nu)w_{\tau\nu}\\
          = \!\!&\!\! (\xi\cdot\tau)^2w_{\tau\tau} + (\xi\cdot\nu)^2 w_{\nu\nu} + v'(x,\xi).
\end{array}
\end{equation}
By the constructions of $v$, we then have
\begin{equation}
\begin{array}{rl}
v(x_0,\xi) = \!\!&\!\! (\xi\cdot\tau)^2v(x_0,\tau)+(\xi\cdot\nu)^2v(x_0,\nu)\\
           \le \!\!&\!\! (\xi\cdot\tau)^2v(x_0,\xi)+(\xi\cdot\nu)^2v(x_0,\nu),
\end{array}
\end{equation}
which leads again to
\begin{equation}
v(x_0,\xi) \le v(x_0,\nu)  \le C(1 + M_2)^\frac{n-2}{n-1}, \quad {\rm on}\ \partial \Omega.
\end{equation}

\vspace{2mm}

{\it Subcase (iii).} $\xi$ is tangential to $\partial \Omega$ at $x_0$. Observing the construction of $v'$, we have $v'(x_0,\xi)=0$. We then have, at $x_0$,
\begin{equation}
\begin{array}{rl}
0 \ge \!\!&\!\! \displaystyle D_\nu v = D_\nu [e^{\frac{\alpha}{2} |Du|^2 + \kappa \Phi}(w_{\xi\xi}-v')] \\
   =  \!\!&\!\! \displaystyle e^{\frac{\alpha}{2} |Du|^2 + \kappa \Phi}[(w_{\xi\xi}-v')D_\nu (\frac{\alpha}{2} |Du|^2 + \kappa \Phi) + D_\nu (w_{\xi\xi}-v')] \\
   =  \!\!&\!\! \displaystyle e^{\frac{\alpha}{2} |Du|^2 + \kappa \Phi}\{[\alpha D_kuD_\nu(D_ku)+\kappa D_\nu \Phi]w_{\xi\xi}+D_\nu u_{\xi\xi}-D_\nu (A_{\xi\xi}+ v') \}\\
   =  \!\!&\!\! \displaystyle e^{\frac{\alpha}{2} |Du|^2 + \kappa \Phi}\{[\kappa D_\nu \Phi + \alpha D_ku(\varphi_k+\varphi_z D_ku - D_iuD_k\nu_i)]w_{\xi\xi}+D_\nu u_{\xi\xi}-D_\nu (A_{\xi\xi}+ v') \}\\
   \ge \!\!&\!\! \displaystyle e^{\frac{\alpha}{2} |Du|^2 + \kappa \Phi}\{(\kappa c_0 - \alpha M)w_{\xi\xi}+D_\nu u_{\xi\xi}-D_\nu (A_{\xi\xi}+ v') \},
\end{array}
\end{equation}
where $c_0=\frac{Ka}{\epsilon_1}$ , $M= \max\limits_{x\in\partial\Omega}|D_ku(\varphi_k+\varphi_z D_ku - D_iuD_k\nu_i)|$. The above inequality gives a relationship between $w_{\xi\xi}(x_0)$ and $D_\nu u_{\xi\xi}(x_0)$, namely
\begin{equation}\label{combine*}
D_\nu u_{\xi\xi} \le -(\kappa c_0 - \alpha M)w_{\xi\xi} +D_\nu (A_{\xi\xi} +v'), \quad {\rm at} \ x_0.
\end{equation}
On the other hand, by tangentially differentiating the boundary condition twice, we obtain
\begin{equation}
D_ku \delta_i\delta_j \nu_k + \delta_iD_ku\delta_j\nu_k + \delta_jD_ku\delta_i\nu_k + \nu_k \delta_i\delta_j D_k u = \delta_i \delta_j \varphi, \quad
{\rm on}\ \partial \Omega.
\end{equation}
Hence at $x_0$, for the tangential direction $\xi$ we have
\begin{equation}\label{combine**}
\begin{array}{rl}
D_\nu u_{\xi\xi} \ge \!\!&\!\! \varphi_z D_{ij}u \xi_i\xi_j - 2 (\delta_i\nu_k)D_{jk}u\xi_i\xi_j + (\delta_i\nu_j)\xi_i\xi_jD_{\nu\nu}u -C\\
                 \ge \!\!&\!\! \varphi_z D_{ij}u \xi_i\xi_j - 2 (\delta_i\nu_k)D_{jk}u\xi_i\xi_j -C\\
                 \ge \!\!&\!\! \varphi_z w_{\xi\xi} - 2 (\delta_i\nu_k)D_{jk}u\xi_i\xi_j -C, \quad {\rm at}\ x_0,
\end{array}
\end{equation}
where the double normal boundary estimate \eqref{3.10} is used in the second inequality. The inequality \eqref{combine**} clearly provides another relationship between $D_\nu u_{\xi\xi}(x_0)$ and $w_{\xi\xi}(x_0)$. Combining this with \eqref{combine*}, we obtain
\begin{equation}\label{combined}
(\kappa c_0 - \alpha M + \varphi_z)w_{\xi\xi} \le 2 (\delta_i\nu_k)D_{jk}u\xi_i\xi_j + D_\nu (A_{\xi\xi} + v') + C, \quad {\rm at}\ x_0.
\end{equation}
Without loss of generality, we can assume the normal at $x_0$ to be $\nu=(0,\cdots,0,1)$, and correspondingly we may assume $\{w_{ij}(x_0)\}_{i,j<n}$ is diagonal with maximum eigenvalue $w_{11}(x_0) > 1$, as in the interior case.
Observing that the first term on the right hand side of  \eqref{combined} only involves tangential second derivatives and using \eqref {2.15}, we can then estimate at $x_0$,
\begin{equation}
\begin{array}{rl}(\kappa c_0 - \alpha M + \varphi_z )w_{\xi\xi} \le \!\!&\!\! C (w_{11} + |DD_\nu u| )\\
                               \le  \!\!&\!\! Cw_{\xi\xi}  + \epsilon M_2 + C_\epsilon.
\end{array}
\end{equation}
We now choose $\kappa$ sufficiently large, such that
\begin{equation}
\kappa\ge \frac{2}{c_0}[\alpha M - \inf \varphi_z - C],
\end{equation}
and again we obtain
\begin{equation}\label {2.47}
v(x_0,\xi) \le C(1 + M_2)^\frac{n-2}{n-1}.
\end{equation}
We now conclude from the above three subcases that if $v$ attains its maximum over $\bar \Omega$ at a point $x_0\in\partial\Omega$, then $v(x_0,\xi)$ is bounded from above as in \eqref{2.47}, which implies the second derivative $D_{\xi\xi}u(x_0)$ is also similarly bounded from above. Combining the above two cases, and using the Cauchy inequality, we obtain the desired estimate \eqref{global C2 bound}  and complete the proof of Theorem \ref{Th1.1}.
\end{proof}

As remarked in Section \ref{Section 1}, we can relax the supersolution hypothesis when $D_{px} A = 0$, that is $A$ is of the form
\eqref{special case}. Moreover the details are then much simpler as we do not need to extend the Pogorelev argument to handle third derivatives. Here we proceed in accordance with Remark 1 in Section 3 of \cite{LTU1986}, assuming as there initially that $B$ is convex with respect to $p$, and replace the auxiliary function $v$ in \eqref{aux fn} by
\begin{equation}
v=v(x, \xi)= w_{\xi\xi}- v' +\frac{\alpha}{2} |Du|^2+ \kappa \Phi,
\end{equation}
where now $\tilde u \in C^2(\bar\Omega)$ in $\Phi=\frac{1}{\epsilon_1}e^{K(\tilde u-u)}$ is an elliptic function with $\tilde u \ge u$ in $\Omega$, as in Lemma \ref{Lemma barrier}.
 In place of \eqref{Lwxixiv'}, we now have the simpler inequality
\begin{equation}
L(w_{\xi\xi}-v') \ge  - C(1+ \mathcal{T} +w_{ii}) .
\end{equation}
And we obtain an estimate from above for $w_{\xi\xi}$ if the maximum of $v$ occurs at an interior
 point of $\Omega$  by taking again sufficiently large constants $\alpha$
 and $\kappa$. If the maximum of $v$ occurs on the boundary $\partial\Omega$, then we proceed as in Case 2 above except now the technical details are simpler and we do not need $D_\nu \Phi \ge 0$ on $\partial\Omega$ but we do need instead $\Omega$ uniformly convex or more generally $\varphi_z  + 2\kappa_1 > 0$, where $\kappa_1$ is the minimum curvature of $\partial\Omega$, to use \eqref{combine**}. We then obtain the estimate \eqref {global C2 bound} as before except that the dependence on $\bar u$ is replaced by a dependence on an elliptic function $\tilde u$. The removal of the condition that $B$ is convex in $p$ can then be addressed in the same way as in \cite{LTU1986} by using Theorem \ref{Th1.2} to construct a supersolution when $B$ is replaced by its infimum and invoking the full strength of Theorem \ref{Th1.1}.

 \subsection*{Remark on Lemma \ref{Lemma barrier}}

 The proof of Lemma \ref{Lemma barrier} following \cite{JT2014, JTY2013} applies very generally. In fact, similarly to Theorem 2.1 in
 \cite{Tru2006}, we may replace the function
 ``$\log\det$'' in \eqref {3.3} by any increasing concave $C^1$ function $f$ on an open convex set $\Gamma$ in the linear space of  $n\times n$ symmetric matrices $\mathbb{S}^n$, which is closed under addition of the positive cone. Here the ellipticity conditions are replaced by the augmented Hessians $Mu(\Omega), M\tilde u(\bar\Omega)\subset\Gamma$, which imply the operator $\tilde F$ is elliptic with respect to $u$ and $\tilde u$ on $\Omega$ and $\bar\Omega$, respectively, and $w^{ij}$ is replaced by $\tilde F_{r_{ij}}$ in the definition of $L$. The general case is covered with a slightly different proof in Section 4 of the forthcoming paper \cite {JT2015}; see also \cite{JTY2014} for the $k$-Hessian case. However for the special case of \eqref{3.3}, the proof of Lemma \ref{Lemma barrier} from \cite{JT2014, JTY2013} may also be simplified somewhat by avoiding the perturbation of $\tilde u$ that is one of the key ingredients of the general argument used there. To see this, we may modify the calculations in the proof of Lemma 2.2 in \cite{JT2014}, with $\epsilon = 0$ and $v=\tilde u - u$, (without using concavity!), to  arrive at the inequality,
 \begin{equation}
 Le^{Kv} \ge Ke^{Kv}\{ w^{ij}[D_{ij}\tilde u - A_{ij}(\cdot,\tilde u,D\tilde u) - w_{ij}] - \eta w^{ii} - D_{p_l}\tilde B(\cdot,u,Du)D_lv\},
 \end{equation}
for any positive constant  $\eta$ and sufficiently large constant $K$ depending also on $\eta$. We then obtain \eqref{key inequality} using the simple inequality
 $$ w^{ij} [D_{ij}\tilde u - A_{ij}(\cdot,\tilde u,D\tilde u)] \ge w^{ii} \lambda[M\tilde u] > 0,$$
 where $\lambda[M\tilde u]$ denotes the minimum eigenvalue of $M\tilde u$, and taking $\eta$ sufficiently small.

\section{Existence and solution estimates}\label{Section 3}
\vskip10pt

In this section we complete the proof of Theorem \ref{Th1.2} and provide alternative conditions for the maximum modulus for solutions of the Neumann problem \eqref{1.1}-\eqref{1.2}. First we formulate a comparison principle for general oblique boundary value problems \eqref{1.7}-\eqref{1.8} with $F$ defined by \eqref{1.9}, with $A$ and $B$ non-decreasing in $z$, and $G\in C^1(\partial \Omega\times\mathbb{R} \times \mathbb{R}^n$,  non-increasing in $z$.
\begin{Lemma}\label{comparison}
Let $u,v\in C^2(\Omega)\cap C^1(\bar \Omega)$ with $\mathcal F$ elliptic, with respect to $u$, in $\Omega$
and $\mathcal G$ oblique with respect to $[u,v]$ on $\partial\Omega$, where  $[u,v] = \{ \theta u + (1-\theta) v: 0 \le\theta\le 1\}$. Assume also that either $G$ is strictly decreasing in $z$ or $A$ or $B$ are strictly increasing in $z$. Then if $\mathcal F[u] \ge \mathcal F[v]$ on the subset of $\Omega$ where $\mathcal F$ is elliptic with respect to $v$ and
$\mathcal G[u] \ge \mathcal G[v]$  on $\partial\Omega$, we have
\begin{equation}\label{u le v}
u \le v, \quad \mbox{in } \ \ \Omega.
\end{equation}
Moreover  if we assume that $\mathcal F$ is elliptic with respect to $[u,v]$ on all of  $\Omega$ , we may relax the strict monotonicity condition on $A,B$ or $G$, provided $u-v$ is not a constant.
\end{Lemma}

The proof of  Lemma \ref{comparison} is standard. By approximating $\Omega$ by a subdomain and approximating
 $u$ by a smaller elliptic function $\underline u$ satisfying $\mathcal F[\underline  u]> \mathcal F[u]$, we infer that the function $u - v$ can only take a positive maximum on the boundary $\partial\Omega$ and $\eqref{u le v}$ then follows from the obliqueness and the strict monotonicity of $G$. When $G$ is only non-increasing in $z$, then we can take
 $\underline u = u -\epsilon( \phi - \min \phi)$ for a defining function $\phi \in C^2(\Omega)\cap C^1(\bar \Omega)$
 such that $\phi = 0$ on  $\partial\Omega$, $\phi < 0$ in $\Omega$ and sufficiently small $\epsilon >0$, to ensure
 $\mathcal G[\underline u] > \mathcal G[u]$ on $\partial\Omega$, whence   a positive maximum of $\underline u - v$ must be taken on in $\Omega$ and we conclude $\eqref{u le v}$ from the strict monotonicity of $F$ with respect to $z$. Note that  when $G$ is strictly decreasing, we need only assume $\mathcal G$ is weakly oblique, that is $G_p\cdot\nu \ge 0$ on $\partial\Omega$ while when $F$ is strictly decreasing we need only assume $\mathcal F$ is degenerate elliptic. In the case when there is no strict monotonicity, the difference $w = u-v$ will satisfy a linear uniformly elliptic differential inequality of the form
$$ {\mathcal{L}}w : =  a^{ij}D_{ij}w + b_iD_iw +cw \ge 0,$$
together with an oblique boundary inequality,
$ \beta\cdot Dw \ge \gamma w$, with coefficients $c\le 0$ and $\gamma\ge 0$, and the result follows from the strong maximum principle and Hopf boundary point lemma; (see\cite{GTbook}).

From Lemma \ref{comparison} we have immediately the uniqueness in Theorem \ref{Th1.2} and the inequality
$\underline u\le u\le\bar u$,
where $\bar u$ and $\underline u$ are the assumed elliptic supersolution \eqref{super equation}-\eqref{super boundary} and subsolution.

Next we obtain a gradient bound for $A$-convex functions for Neumann problem \eqref{1.1}-\eqref{1.2}, where $A$ satisfies a quadratic bound from below, \eqref{QS}, by a modification of our argument for the Dirichlet problem in \cite{JTY2013}. For this purpose, we  formulate the following gradient estimate as a lemma.

\begin{Lemma}\label{gradient bound}
Let $u\in C^2(\Omega)\cap C^1(\bar \Omega)$ satisfy
\begin{equation}\label{weak convexity}
D^2u\ge -\mu_0(1+|Du|^2) I,
\end{equation}
in a $C^2$ domain  $\Omega\subset\mathbb{R}^n$, with
\begin{equation}\label{lower for boundary normal}
D_\nu u \ge -\sigma,
\end{equation}
on $\partial\Omega$, where $\mu_0$ and $\sigma$   are non-negative constants. Then we have the estimate
\begin{equation}\label{grad bound estimate}
|Du|\le C,
\end{equation}
where $C$ depends on $\mu_0, \sigma, \Omega$ and $\sup |u|$.
\end{Lemma}

\begin{proof}
Defining $\tilde u = u - \sigma \phi$, where as in Section \ref{Section 2}, $\phi\in C^2(\bar\Omega)$ is a negative defining function for $\Omega$ satisfying $D_\nu \phi = -1$ on $\partial \Omega$, we see that $\nu\cdot D \tilde u \ge 0$ on $\partial \Omega$.
Consequently at a maximum point $x_0\in \bar \Omega$ of the function
\begin{equation}\label{gradient auxi}
w= e^{\kappa\tilde u}|D\tilde u|^2,
\end{equation}
we have
\begin{equation}\label{property on max}
D\tilde u\cdot Dw \le 0.
\end{equation}
From \eqref{weak convexity}, we have
\begin{equation}\label{lower qb for tilde u}
\begin{array}{rl}
D^2\tilde u \!\!&\!\!\displaystyle = D^2 u - \sigma D^2 \phi \\
        \!\!&\!\!\displaystyle \ge -\mu_0(1+|Du|^2)I - \sigma \Lambda_\phi I\\
        \!\!&\!\!\displaystyle \ge -\mu_0(1+2|D\tilde u|^2 + 2\sigma^2|D\phi|^2)I - \sigma \Lambda_\phi I\\
        \!\!&\!\!\displaystyle \ge -\mu_1(1+|D\tilde u|^2)I,
\end{array}
\end{equation}
for some positive constant $\mu_1$ depending on $\mu_0$, $\sigma$, $D\phi$ and $\Lambda_\phi$, where $\Lambda_\phi$ denotes the maximum eigenvalue of the Hessian matrix of $\phi$ and depends on the domain $\Omega$.
With the lower quadratic bound \eqref{lower qb for tilde u} for the Hessian matrix $D^2\tilde u$ in hand, by choosing the constant $\kappa$ sufficiently large as in Section 4, \cite{JTY2013}, we can obtain from \eqref{property on max},
\begin{equation}\label{gradient tilde u}
|D\tilde u| \le C,
\end{equation}
at $x_0$, where the constant $C$ depends on $\mu_0$, $\sigma$ and $\Omega$. We then conclude a global gradient estimate from  \eqref{gradient tilde u} and the construction of $\tilde u$,
\begin{equation}\label{gradient}
|Du| \le C,
\end{equation}
where $C$ depends on $\mu_0$, $\Omega$, $\sigma$ and $\sup |u|$.
\end{proof}
We remark that by taking more careful account of the constant dependence in the proof of Lemma \ref{gradient bound}
we infer a sharper estimate
\begin{equation}\label{sharper estimate}
|Du| \le C(1+\sigma),
\end{equation}
where $C$ depends on $\mu_0$, $\Omega$ and $\sup |u|$.

Note that the gradient estimate \eqref{grad bound estimate} in Lemma \ref{gradient bound} and the sharper gradient estimate \eqref{sharper estimate} hold for any solution $u$ satisfying the weak convexity condition \eqref{weak convexity} and the lower bound condition \eqref{lower for boundary normal} for normal derivative on the boundary. We now apply Lemma \ref{gradient bound} to obtain the gradient estimate for $A$-convex solutions of the Neumann problem \eqref{1.1}-\eqref{1.2} with $A$ satisfying the lower quadratic bound \eqref{QS}. From the $A$-convexity of the solution $u$ and the quadratic structure condition \eqref{QS}, the solution $u$ satisfies the weak convexity condition \eqref{weak convexity}. The Neumann boundary condition \eqref{1.2} provides us a lower bound $D_\nu u \ge \inf_{\partial \Omega}\varphi(x,u)$. Applying Lemma \ref{gradient bound}, we then obtain the global gradient estimate for Neumann problem \eqref{1.1}-\eqref{1.2}, that is $|Du| \le C$ for $C$ depending on $\mu_0$, $\Omega$, $\varphi$ and $\sup|u|$.

Since we now have obtained the derivative estimates up to second order, we can use the continuity method to prove our existence theorem.
\begin{proof}[Proof of Theorem \ref{Th1.2}.]
From the second derivative estimate, Theorem \ref{Th1.1} and the preceding solution and gradient estimates we can derive a global second derivative H\"{o}lder estimate
\begin{equation}\label{global Holder}
|u|_{2,\alpha;\Omega}\le C,
\end{equation}
for elliptic solutions $u\in C^4(\Omega)\cap C^3(\bar \Omega)$ of the semilinear Neumann boundary value problem \eqref{1.1}-\eqref{1.2} for $0<\alpha<1$. The estimate \eqref{global Holder} is obtained in \cite{LT1986}, Theorem 3.2, (see also \cite{LieTru1986,Tru1984}). With this $C^{2,\alpha}$ estimate, we can use the method of continuity, Theorem 17.22 and Theorem 17.28 in \cite{GTbook}, to derive the existence of a solution $u\in C^{2,\alpha}(\bar \Omega)$, using the supersolution $\bar u$ as an initial solution. To be rigorous, we should assume that $A$ and $B$ are $C^{2,\alpha}$ smooth, $\varphi$ is $C^{3,\alpha}$ smooth and $\Omega\in C^{4,\alpha}$ for some $\alpha>0$ to get a solution
$u\in C^{4,\alpha}(\bar\Omega)$ by the Schauder theory, (see \cite{GTbook}, Section 6.7), and then by approximation get a solution
$u\in C^{3,\alpha}(\bar\Omega)$. Alternatively we can use the Aleksandrov-Bakel'man maximum principles (see \cite{GTbook}, Theorem 9.1, Theorem 9.6) to carry over the proof of Theorem \ref{Th1.1} to solutions $u\in W^{4,n}(\Omega)\cap C^3(\bar \Omega)$ and use $L_p$ regularity as well, (\cite{GTbook}, Section 9.5) to improve $C^{2,\alpha}(\bar \Omega)$ solutions with $0<\alpha<1$ to be in the Sobolev spaces $W^{4,p}(\Omega)\cap C^{3,\delta}(\bar \Omega)$ for all $p<\infty$, $0<\delta<1$.
\end{proof}

In the rest of this section we will consider more explicit conditions for solution bounds. Here we consider the oblique boundary value problems \eqref{1.7}-\eqref{1.8} with $F$ defined by \eqref{1.9} and $G$ defined by \eqref{G in oblique form}, that is  the Monge-Amp\`ere type equation \eqref{1.1} together with the oblique boundary condition
\begin{equation}\label{oblique bvp semilinear form}
D_\beta u = \varphi(x,u),\quad \mbox{on}\ \  \partial \Omega.
\end{equation}
First we note that we also obtain bounds for solutions $u$ of \eqref{1.1}-\eqref{1.2} if $ \bar u$ and  $\underline u$ are only assumed to be supersolutions and subsolutions, without any assumed boundary conditions, provided we strengthen the monotonicity of $\varphi$. In particular we may assume, as in \cite{LTU1986},   there exists a positive constant $\gamma_0$ such that
\begin{equation}\label{uniform monotonicity}
\varphi_z(x,z) \ge \gamma_0
\end{equation}
for all $(x, z)\in\partial\Omega\times\mathbb{R}$. In the light of Lemma \ref{comparison}, we may interpret a supersolution as satisfying \eqref{super equation} only at points of ellipticity.  Since $A$ and $B$ are non-decreasing,  supersolutions  and elliptic subsolutions  are preserved under addition and subtraction respectively of positive constants. Accordingly, by subtracting a positive constant from $\underline u$ and using \eqref{uniform monotonicity} we can assume $D_\beta \underline u \ge \varphi(x,\underline u)$ on $\partial \Omega$, whence $u\ge \underline u$ in $\Omega$. Similarly by adding a positive constant to
$\bar u$ we obtain $D_\beta \bar u \le \varphi(x,\bar u)$ on $\partial \Omega$, so that $u\le \bar u$ in $\Omega$. Note that for this argument we may  replace \eqref{uniform monotonicity} by the weaker conditions
\begin{equation}\label{3.14}
{ ({\rm sign}z)}\varphi (\cdot,z) \rightarrow \infty, \ {\rm as}\ |z|\rightarrow \infty.
\end{equation}
The conditions \eqref {uniform monotonicity}, \eqref{3.14} may be further weakened when constants are subsolutions or supersolutions.  We first consider the bound from below,
 under the following conditions:
\begin{equation}\label{alternative SC1+}
A(x,z,0)\le 0, \ \det[-A(x,z,0)] > B(x,z,0) ,\quad {\rm for \  all}\ x\in \Omega,\  z<- K,
\end{equation}
\begin{equation}\label{alternative SC2+}
\varphi(x,z) <0 ,\quad {\rm for \  all}\ x\in \partial\Omega, \ z < -K,
\end{equation}
where $K$ is a positive constant. Under the assumptions \eqref{alternative SC1+} and \eqref{alternative SC2+}, we can readily obtain the solution bound as follows. Suppose $u$ attains its minimum over $\bar \Omega$ at a point $x_0$ and
$u(x_0) < -K$. If $x_0\in \Omega$, we have $Du(x_0)=0$, $D^2u(x_0)\ge 0$. From the equation \eqref{1.1}, we have $\det [-A(x_0,u(x_0),0)]-B(x_0,u(x_0),0)\le 0$ so that by \eqref{alternative SC1+}, we must have $u(x_0)\ge- K$. If $x_0\in\partial\Omega$, we have $D_\beta u(x_0)\ge 0$. From the oblique boundary condition, we have $\varphi(x_0,u(x_0))\ge 0$. By \eqref{alternative SC2+}, we again have $u(x_0)\ge -K$. Note that condition \eqref{alternative SC1+} implies sufficiently small constants are subsolutions of the oblique boundary value problem \eqref{1.1}-\eqref{oblique bvp semilinear form} thereby providing lower solution bounds, by the comparison principle, Lemma \ref{comparison}.  Therefore the subsolution assumption in Theorem \ref{Th1.2}  can be replaced by  the structure conditions \eqref{alternative SC1+} and \eqref{alternative SC2+}, with  $\min\underline u$ replaced by $-K$ in $\mathcal I$. We also remark that condition  \eqref{alternative SC1+} follows from a uniform monotonicity condition on $A$, namely
\begin{equation}\label{uniform monotonicity on A}
D_zA_{ij}(x,z,p)\xi_i\xi_j \ge \gamma_1 |\xi|^2,
\end{equation}
for all $(x,z,p)\in \Omega \times \mathbb{R} \times \mathbb{R}^n$, $\xi\in \mathbb{R}^n$ and some $\gamma_1>0$,
which is a stronger form of the A4w condition used for generated prescribed Jacobian equations in geometric optics in \cite{JT2014, Tru2014}, together with $B$ being non-decreasing in $z$.

In this sense, the condition \eqref{alternative SC1+} is a weakening of the uniform monotonicity of $A$, while the condition \eqref{alternative SC2+} is a weakening of the uniform monotonicity of $\varphi$. On the other hand, condition \eqref{alternative SC1+} is restrictive in that it excludes the case when $A$ is independent of $z$, which occurs in optimal transportation.

Corresponding conditions also provide bounds from above. Here though the analogue of \eqref{alternative SC1+} is more general, namely
\begin{equation}\label{alternative SC1-}
\det[-A(x,z,0)]  < B(x,z,0) ,\quad {\rm for \  all}\ x\in \Omega,\  z>K, \  A(x,z,0) < 0,
\end{equation}
while instead of \eqref{alternative SC2+}, we have
\begin{equation}\label{alternative SC2-}
\varphi(x,z) >0 ,\quad {\rm for \  all}\ x\in \partial\Omega, \ z > K,
\end{equation}
where $K$ is a positive constant. Note that condition \eqref{alternative SC1-} extends the condition in Section 4 of \cite{JTY2013}, namely that the maximum eigenvalue of $A(x,z,0)$ is non-negative for all $x\in \Omega$, $z>K$ for some positive constant $K$ and implies that constants larger than $K$ will be supersolutions, where they are elliptic.

To complete this section, we derive a  lower bound for optimal transportation equations
and present the corresponding existence result.
\subsection*{Optimal transportation equations}

In the optimal transportation case, we can replace the existence of a subsolution in Theorem \ref{Th1.2} by an extension of the sharp conditions (1.4), (1.5) in \cite{LTU1986}, through an extension of the Aleksandrov-Bakel'man estimate in Theorem 2.1 of \cite{LTU1986}. Optimal transportation equations are special cases of prescribed Jacobian equations where the mapping $Y$ is generated by a cost function $c$ defined on a domain $\mathcal D\subset\mathbb{R}^n\times\mathbb{R}^n$. We assume
$\bar\Omega\times\bar\Lambda\subset\mathcal D$, for some domain $\Lambda\subset\mathbb{R}^n$, and $c\in C^2(\mathcal D)$
 satisfies the conditions, (from \cite{MTW2005}):

\begin{itemize}
\item[{\bf A1}:]
For each $x\in \Omega$, the mapping $c_x(x,\cdot)$ is one-to-one in $y \in \mathcal D^*_x = \{y\in \mathbb{R}^n\big{|}(x,y)\in \mathcal D\}$;

\item[{\bf A2}:]
$\det c_{x,y}\neq 0$ on $\mathcal D$.
\end{itemize}
Then the mapping $Y$ is given by
\begin{equation}\label{Y OT case}
Y(x,p)=c_x^{-1}(x,\cdot)(p)
\end{equation}
and is well defined for $p\in \mathcal U_x=\{p \in \mathbb{R}^n \big{|}\  p=c_x(x,y)$ for some $y \in \mathcal D^*_x\}$. In the resultant Monge-Amp\`ere type equation, we then have from \eqref{def:AB},
\begin{equation}\label{AB OT case}
A(x,z,p) = A(x,p)=c_{xx}(x,Y(x,p)), \ \ \ B=|\det c_{x,y}|\psi,
\end{equation}
and equation \eqref{1.1} is well defined for solutions $u$ which are $A$-convex and satisfy $Du(x)\in \mathcal U_x$, for each $x\in \Omega$. We call such solutions admissible. In the optimal transportation case, $c$-affine functions, that is functions of the form $\bar u=c(x,y)+c_0$, for constant $c_0$ and $(\Omega,\{y\})\subset \mathcal D$ are automatically supersolutions as they satisfy the homogeneous equation
\begin{equation}\label{equation OT case}
\det (D^2\bar u - A(x,D\bar u))=0,
\end{equation}
and hence provide upper bounds for solutions of (weakly) oblique boundary value problems,
\begin{equation}\label{oblique bvp OT case}
D_\beta u=\varphi(x,u),\quad \mbox{on } \ \partial\Omega,
\end{equation}
where $\beta\cdot\nu\ge 0$ on $\partial \Omega$, under a uniform monotonicity condition \eqref{uniform monotonicity}. For lower bounds we impose a structure condition
\begin{equation}\label{psi OT case}
\psi(x,z,p)\le \frac{f(x)}{f^* \circ Y(x,p)}
\end{equation}
for all $x\in \Omega$, $z\le m_0$, $Y(x,p)\in \Lambda$, where $f\ge 0,\in L^1(\Omega)$, $f^*>0,\in L^1_{loc}(\Lambda)$ satisfy
\begin{equation}\label{sharp condition OT case}
\int_\Omega f < \int_\Lambda f^*
\end{equation}
and $m_0$ is a constant.

We now have the lower solution bound in the optimal transportation case.

\begin{Lemma}\label{Lemma 3.3}
Let $u\in C^2(\Omega)\cap C^1(\bar \Omega)$ be an admissible solution of equation \eqref{PJE}, in the optimal transportation case \eqref{Y OT case}, with cost function $c$ satisfying A1, A2. Suppose that $\psi$ satisfies \eqref{psi OT case} and
\begin{equation}\label{boundary ineq}
D_\beta u \le \gamma_0 u +\varphi_0 \quad \mbox{on } \partial \Omega,
\end{equation}
 for $u\le m_0$, where $\beta \in L^\infty(\partial\Omega)$, $\beta \cdot \nu \ge 0$ on $\partial \Omega$ and $\gamma_0 > 0$ and $\varphi_0 \ge 0$ are constants. Then we have the lower bound
\begin{equation}
u\ge -C, \quad \mbox{in}\ \  \Omega,
\end{equation}
where $C$ is a positive constant depending on $\Omega, f, f^*, \beta, \gamma_0, \varphi_0$ and $c$.
\end{Lemma}
\begin{proof}
Our proof is adapted from the second author's 2004 Singapore Institute of Mathematical Sciences lectures and the case where $c(x,y)=x\cdot y$, that is $Y=p$ and $A=0$, in \cite{LTU1986}. First, we note that if we have a global support from below at a point $x_0\in \Omega$, that is
\begin{equation}\label{global support OT case}
u(x)\ge u(x_0)+c(x,y_0)-c(x_0,y_0)
\end{equation}
for all $x \in \Omega$, then we must have $y_0=Y(x_0,Du(x_0))$.
Defining $T=Y(\cdot, Du)$, we have by \eqref{PJE}, \eqref{psi OT case} and the change of variable formula
\begin{equation}
\begin{array}{rl}
\displaystyle\int_\Omega f \!\!&\!\!\displaystyle \ge \int_\Omega |\det DT| f^*\circ T\\
                           \!\!&\!\!\displaystyle \ge \int_{T(\Omega_0)}f^*
\end{array}
\end{equation}
where $\Omega_0=\{x\in \Omega \big{|}\  u(x)<m_0\}$. Hence by our condition \eqref{sharp condition OT case} on $f$ and $f^*$, there exists a point $y_0\in \Lambda-T(\Omega_0)$. It then follows by upward vertical translation of a $c$-affine lower bound, that there exists a point $x_0\in \partial\Omega_0$ such that
\begin{equation}\label{global support OT case}
u(x)\ge u(x_0)+c(x,y_0)-c(x_0,y_0)
\end{equation}
 for all $x\in \Omega$. If $x_0\in \partial \Omega$, we must also have
\begin{equation}
D_\beta u(x_0) \ge D_\beta c(x_0,y_0)
\end{equation}
whence by the boundary inequality \eqref{boundary ineq},we obtain
\begin{equation}
u(x_0) \ge \frac{1}{\gamma_0} [D_\beta c(x_0,y_0) - \varphi_0].
\end{equation}
If $x_0 \not \in \partial\Omega$, then we must have $u(x_0)=m_0$. Hence by \eqref{global support OT case} again, we obtain for $x_0\in \partial \Omega$
\begin{equation}\label{lower u bound x on boundary}
\begin{array}{rl}
u(x) \!\!&\!\! \displaystyle \ge u(x_0) + c(x,y_0)- c(x_0,y_0) \\
     \!\!&\!\! \displaystyle \ge  \frac{1}{\gamma_0} [D_\beta c(x_0,y_0) - \varphi_0] + c(x,y_0)- c(x_0,y_0)  \\
     \!\!&\!\! \displaystyle \ge  - \frac{\varphi_0}{\gamma_0} -(\frac{|\beta|}{\gamma_0}+{\rm diam} \Omega)\sup_\Omega |c_x(\cdot,y_0)|
\end{array}
\end{equation}
while for $x_0\not \in \partial\Omega$ we obtain
\begin{equation}\label{lower u bound x not on boundary}
u(x) \ge m_0 -{\rm diam} \Omega \sup_\Omega |c_x(\cdot, y_0)|.
\end{equation}
To remove the dependence on $y_0$ in \eqref{lower u bound x on boundary} and \eqref{lower u bound x not on boundary}, we may consider an exhaustion of $\Lambda$, say by defining subdomains
\begin{equation}
\Lambda_R = \{y\in \Lambda \big{|} \ |y|<R, {\rm dist}(y,\partial \Lambda)> \frac{1}{R} \}
\end{equation}
for $R\ge 1$. Then by \eqref{sharp condition OT case}, we have
\begin{equation}
\int_\Omega f = \int_{\Lambda_R} f^*
\end{equation}
for some sufficiently large $R$, and we obtain from \eqref{lower u bound x on boundary} and \eqref{lower u bound x not on boundary}, the estimate,
\begin{equation}
u(x) \ge  \min\{m_0,-\frac{\varphi_0}{\gamma_0}\}  - (\frac{|\beta|}{\gamma_0}+ {\rm diam}\Omega) \sup_{\Omega\times \Lambda_R} |Dc|.
\end{equation}
This completes the proof of Lemma \ref{Lemma 3.3}.
\end{proof}

As a corollary of Lemma \ref{Lemma 3.3} and the proof of Theorem \ref{Th1.2}, we then have the following variant of Theorem \ref{Th1.2} in the optimal transportation case. For this purpose we note that the boundary condition
\eqref {1.2} and the monotonicity condition \eqref {uniform monotonicity} imply \eqref {boundary ineq} with $\beta = \nu$ and

$$ \varphi_0 = - \gamma_0 m_0 + \sup_{\partial\Omega} \varphi(\cdot,m_0). $$

\begin{Corollary}\label{Cor 3.1}
 Suppose that equation \eqref{1.1} is a prescribed Jacobian equation of the form \eqref{PJE} generated by a cost function $c\in C^2(\mathcal{D})$ satisfying conditions A1 and A2 and $\mathcal U_x=\mathbb{R}^n$ for all $x\in \Omega$, with $\psi$ satisfying the structure conditions \eqref{psi OT case}, \eqref{sharp condition OT case}. Let $A$, $B$, $\varphi$ and $\Omega$ satisfy the hypotheses of Theorem \ref{Th1.2} except for the existence of an elliptic subsolution, with $\varphi$ satisfying \eqref{uniform monotonicity} and $\Omega$ assumed to be uniformly $A$-convex with respect to $\varphi$ and $-C$, that is  \eqref{A convexity}  holds for
 $p\cdot \nu \ge \varphi(\cdot, -C)$ on $\partial\Omega$, where $C$ is the constant in Lemma 3.3.
Then the Neumann boundary value problem \eqref{1.1}-\eqref{1.2} has a unique elliptic solution $u\in C^{3,\alpha}(\bar \Omega)$ for any $\alpha <1$.
\end{Corollary}

We remark that as in \cite{LTU1986},  condition \eqref{sharp condition OT case} is necessary for an elliptic solution
$u\in C^2(\Omega)\cap C^{0,1}(\bar\Omega)$ of \eqref{PJE}, with $Du(x) \in \mathcal U_x$ for all $x\in \bar\Omega$. 

In accordance with our remarks following the statement of Theorem \ref{Th1.2}, pertaining to the special case
\eqref{special case}, and using the argument at the end of Section \ref{Section 2}, we can remove the supersolution condition in Corollary \ref{Cor 3.1} for convex domains. To apply the argument at the end of Section \ref{Section 2}, we also need to use the existence of an elliptic function, as provided by Lemma 2.1 in \cite{JT2014}. In this way, we obtain an extension of Theorem 1.1 in \cite{LTU1986}, which corresponds to the special case $c(x,y) = x\cdot y$, (or equivalently, the case $c(x,y)=- |x-y|^2/2$). Note that the matrix $A$ generated by the cost function satisfies \eqref{special case} when the cost $c=c(x-y)$. Examples of regular and strictly regular cost functions are given in \cite{TruWang2009} and \cite{LiuTru2010}. However most of these examples do not satisfy $\mathcal U_x = \mathbb{R}^n$ and in general we need additional controls on gradients to prove classical existence theorems.

We also remark that Lemma \ref{Lemma 3.3}
and Corollary \ref{Cor 3.1} are readily extended to generated prescribed Jacobian equations \cite{Tru2014}.

\section{Oblique boundary value problems }\label{Section 4}
\vskip10pt

In this section we consider more general oblique boundary value problems for Monge-Amp\`ere type equations under the hypothesis that the matrix function $A$ is strictly regular. As remarked in Section \ref{Section 1}, this condition also leads to a much simpler proof in the Neumann case. Also we do not need to restrict to semilinear problems of the form \eqref{G in oblique form} but can consider nonlinear boundary conditions of the general form \eqref{1.8}, where $G$ is also concave with respect to $p$. Our approach is already indicated in Section 4 of \cite{TruWang2009} and we will carry over some of the basic details from there. Moreover our results can also be seen as special cases of those for  general augmented Hessian equations in \cite{JT2015}.
For second derivative estimates, we will assume that the function
$G\in C^2(\partial\Omega\times\mathbb{R}\times\mathbb{R}^n)$ is oblique with respect to a solution $u$,
that is from \eqref {oblique},
\begin{equation}\label{obliqueness estimate}
G_p(\cdot,u,Du)\cdot\nu \ge\beta_ 0, \quad \mbox{on } \  \partial\Omega,
\end{equation}
for a positive constant $\beta_0$, and is concave in $p$, with respect to $u$, in the sense that
\begin{equation}\label{convexity}
G_{pp}(\cdot,u, Du) \le 0,  \quad \mbox{on } \  \partial\Omega.
\end{equation}
We now have the following extension and improvement of Theorem \ref{Th1.1} in the strictly regular case.

\begin{Theorem}\label{Th4.1}
Let $u\in C^4(\Omega)\cap C^{3}(\bar \Omega)$ be an elliptic solution of the boundary value problem \eqref{1.1}-\eqref{1.8} in a $C^{3,1}$  domain $\Omega\subset\mathbb{R}^n$, which is uniformly $A$-convex with respect to $G$ and $u$, where $A\in C^2(\bar \Omega\times \mathbb{R}\times \mathbb{R}^n)$ is strictly regular in $\bar\Omega$, $B>0,\in C^2(\bar \Omega\times \mathbb{R}\times \mathbb{R}^n)$  and $G\in C^{2,1}(\partial\Omega\times\mathbb{R}\times\mathbb{R}^n)$ satisfies \eqref{obliqueness estimate} and \eqref{convexity}.
Then we have the estimate
\begin{equation}\label{oblique C2 bound}
\sup\limits_\Omega |D^2 u|\le C,
\end{equation}
where $C$ is a constant depending on $n, A, B, G, \Omega, \beta_0$ and $|u|_{1;\Omega}$.
\end{Theorem}

\begin{proof}
As in the proof of Theorem \ref{Th1.1}, we first consider the estimation of the nontangential second derivatives. In the semilinear case \eqref {G in oblique form}, we can simply replace $\nu$ by $\beta$ there and deduce in place  of  \eqref{2.15}, the estimate
\begin{equation}\label{4.4}
|D_{\beta\xi} u|\le C(1 + M_2)^\frac{n-2}{n-1}, \quad {\rm on}\ \partial \Omega,
\end{equation}
for any direction $\xi$, where as in Section \ref{Section 2}, $M_2 = \sup_\Omega |D^2u|$. In the general case, we have the same estimate \eqref{4.4}, from the estimate (4.4) in \cite{TruWang2009},
where now $\beta = G_p(\cdot,u,Du)$. Now differentiating the boundary condition \eqref{1.8} twice with respect to a tangential $C^2$ vector field $\tau$ we obtain as in the estimate (4.10)  in \cite {TruWang2009},
\begin{equation}\label{third deriv}
\begin{array}{rl}
u_{\tau\tau\beta} \!\!&\!\! \ge -D_{p_kp_l}G u_{k\tau}u_{l\tau} - C(1 +M_2)\\
                     \!\!&\!\! \ge - C(1 +M_2), \quad {\rm on}\ \partial \Omega,
\end{array}
\end{equation}
by virtue of the  concavity of $G$ with respect to $p$. For convenience we write here
$ u_{i\tau} = u_{ij}\tau_j$, $u_{\tau\tau} = u_{ij}\tau_i\tau_j$, $u_{\tau\tau\beta} = u_{ijk}\tau_i\tau_j\beta_k$.
To handle the pure tangential derivatives we extend the $C^2$ vector field $\tau$ to all of $\bar\Omega$ and set
\begin{equation}
v=w_{\tau\tau} -K(1+M_2)\phi,
\end{equation}
where as in the proof of Theorem \ref{Th1.1}, $\phi\in C^2(\bar\Omega)$ is a negative defining function for $\Omega$ satisfying $D_\nu \phi = -1$ on $\partial \Omega$ and $K$ is a constant
such that
\begin{equation}
 D_\beta [w_{ij} \tau_i\tau_j] >-K(1+M_2)\beta_0, \quad {\rm on}\ \partial \Omega.
\end{equation}
In particular we may fix $\tau$ with $\tau_i = x_i - (x\cdot\nu) \nu_i$, $i = 1, \ldots, n$, where as in Section \ref{Section 2}, $\nu$ is a smooth extension of the inner normal $\nu$ to $\bar\Omega$.
It then follows that $D_\beta v > 0$ on  $\partial \Omega$ so that $v$ must take its maximum on $\bar\Omega$ at an interior point $x_0\in\Omega$, with  $\mathcal Lv(x_0) \le 0$.
Now we can adapt the proof of the interior second derivative estimate in \cite{MTW2005} and \cite{TruWang2008}, differentiating the equation \eqref{1.1}, in the form \eqref{3.3}, twice with respect to $\tau$ and using also the concavity of the function ``$\log\det$'',  together with  \eqref{third deriv} to control $K$, to estimate at $x_0$,
\begin{equation}\label{pre-est}
w^{ij}A_{ij,kl}u_{k\tau}u_{l\tau} \le C[(1+ M_2)w^{ii} + |Du_\tau|^2].
\end{equation}
We note that when we twice differentiate  \eqref{1.1} with respect to a variable vector field $\tau$, to calculate
$\mathcal L v$, we  encounter terms arising from derivatives of $\tau$ which are not present in the constant case \eqref{twice differentiated}. Apart from the terms in third derivatives these can be directly estimated by $C(1+ M_2)w^{ii}$. Retaining the third derivative terms,  we would supplement the right hand side of \eqref{pre-est}, by
\begin{equation}\label{4.9}
\begin{array}{rl}
- w^{ik}w^{jl}\!\!&\!\!D_{\tau}w_{ij}D_{\tau}w_{kl}  + 4 w^{ij} D_i\tau_kD_\tau w_{jk} \\

              = \!\!&\!\! -w^{ik}w^{jl}D_{\tau}w_{ij}D_{\tau}w_{kl} + 4 w^{ik}w^{jl} w_{jk} D_i\tau_kD_\tau w_{kl}\\

              \le \!\!&\!\!-w^{ik}w^{jl}(D_\tau w_{ij} -2w_{jk} D_i\tau_k)(D_{\tau}w_{kl} - 2w_{il}D_j\tau_l) + 4w^{ik}w^{jl}w_{il}D_j\tau_lw_{jk} D_i\tau_k\\

              \le \!\!&\!\!4(D_i\tau_i)^2
\end{array}
\end{equation}
so that the estimate \eqref{pre-est} is unaffected. To use the strictly regular condition,
\begin{equation}
A_{ij,kl}\xi_i\xi_j\eta_k\eta_l \ge c_0|\xi|^2|\eta|^2,
\end{equation}
for all $\xi, \eta \in \mathbb{R}^n$ satisfying the orthogonality $\xi\perp \eta$, where $c_0$ is  a positive constant depending on $A$ and $|u|_{1;\Omega}$,
we choose coordinates so that $w$ is diagonalised at $x_0$, so that
\begin{equation}\label{4.11}
\begin{array}{rl}
w^{ij}A_{ij,kl}w_{k\tau}w_{l\tau} \!\!&\!\!\displaystyle = w^{ii}A_{ii,kl}w_{kk}w_{ll}\tau_k\tau_l\\
                                  \!\!&\!\!\displaystyle \ge\sum\limits_{k,l\ne i} w^{ii}A_{ii,kl}w_{kk} w_{ll}\tau_k\tau_l - CM_2\\
                                  \!\!&\!\!\displaystyle \ge c_0w^{ii} \sum(w_{kk}\tau_k)^2  - CM_2.
 \end{array}
 \end{equation}
Hence we obtain from  \eqref{pre-est}, \eqref{4.11} and \eqref{2.11},
 \begin{equation} \label{4.12}
D_{\tau\tau} u(x_0) \le C(1 +M_2)^\frac{1}{2} .
\end{equation}
At this point we need to return to our choice of $\phi$ to ensure that $\inf \phi \ge -\epsilon$ for some small positive constant $\epsilon$. This can be done for example by mollification of the function $ -\inf\{d,\epsilon\}$ for sufficiently small $\epsilon$, where the constant $C = C_\epsilon$ in \eqref{4.12} will depend also on $\epsilon$. Alternatively, we may simply restrict to a boundary strip $\Omega_\epsilon = \{\phi > - \epsilon\}$ and use the interior second derivative estimates \cite{MTW2005,TruWang2009} to estimate $v$ on the inner boundary $\{\phi = -\epsilon\}$.  Accordingly we obtain from \eqref{4.12},
$$ v(x_0) \le C_\epsilon(1 +M_2)^\frac{1}{2}  + \epsilon M_2$$
and hence we get an estimate
\begin{equation}\label{tangential deriv}
 D_{\tau\tau} u \le C_\epsilon(1 +M_2)^\frac{1}{2} + \epsilon M_2 \quad {\rm on}\ \partial \Omega.
 \end{equation}
Since for any direction $\xi$, we  have,
 \begin{equation}
u_{\xi\xi}=u_{\tau\tau}
    +b(u_{\tau\beta}+u_{\beta\tau})+b^2u_{\beta\beta},
\end{equation}
where
$$ b=\frac {\xi\cdot\gamma}{\beta\cdot \gamma},
 \ \ \ \tau=\xi-b\beta,$$
 we then obtain a boundary estimate in the form,
 \begin{equation}\label{bdy est}
\sup\limits_{\partial\Omega} |D^2 u|\le \epsilon M_2 + C_\epsilon,
\end{equation}
for any sufficiently small $\epsilon > 0$, by combining \eqref{4.4} and  \eqref{tangential deriv}.
The global second derivative estimate \eqref{oblique C2 bound} now follows from the global second derivative estimates in  \cite{TruWang2008,TruWang2009} by choosing $\epsilon$ sufficiently small.
 \end{proof}

 The details in the proof of Theorem \ref{Th4.1} can be further varied. For example we can replace $v$ by
\begin{equation}\label{alt v}
 v=(1-K\phi)w_{\tau\tau} ,
 \end{equation}
 for a sufficiently large constant $K$, where $\phi$ is the same negative defining function as in the proof of Theorem \ref{Th4.1}. As remarked in Section \ref{Section 1}, we also obtain a much simpler proof of Theorem \ref{Th1.1} in the strictly regular case, without need for the supersolution and monotonicity hypotheses. Moreover
 by flattening the boundary $\partial\Omega$ in a neighbourhood $\mathcal N$ of a fixed point
 $x_1\in \partial \Omega$, we can localise the second derivative estimate by modifying \eqref{alt v}
 \begin{equation}
 v=\eta(1-K\phi)w_{\tau\tau}
\end{equation}
where $\eta$ is a suitable cut-off function satisfying $D_\nu\eta = 0$ on $\mathcal N\cap\partial \Omega$.
 Accordingly, we obtain for any ball $B=B_R(x_0)$ of radius $R >0$
and centre $x_0$, the local estimate
\begin{equation}\label{local C2 bound}
 |D^2u(x_0)| \le  \frac{C}{R^2},
\end{equation}
 for elliptic solutions $u\in C^4(B\cap\Omega)\cap C^{3}(B\cap\bar \Omega)$ of \eqref{1.1} satisfying \eqref{1.2}
 on $B\cap\partial\Omega$, where $B\cap\partial\Omega$ is uniformly $A$-convex
  with respect to $G$ and $u$ in the sense that
 $$  (D_i\nu_j -D_{p_k}A_{ij}(\cdot,u,Du)\nu_k)\tau_i\tau_j \le-\delta_0$$
on $B\cap\partial\Omega$ for $G(x,u,Du)\ge 0$ and any unit tangential vector $\tau$ and a positive constant $\delta_0$. The constant $C$ in \eqref{local C2 bound} depends on $n, A, B, \Omega, \delta_0, \phi$ and $|u|_{1;\Omega}$.
We also point out that comparability of differentiation with respect to a general vector field and a constant vector field
in the proof of \ref{Th4.1}, which follows from the identity \eqref{4.9}, is special to the Monge-Amp\`ere case. A  different and more detailed proof of the critical tangential estimate \eqref{tangential deriv} is  provided for more general augmented Hessian equations in \cite{JT2015}, Lemma 2.3.

Returning to the example from conformal geometry in Section \ref{Section 1}, namely \eqref {transformed Yamabe},
\eqref {Euclidean Yamabe} with $\mathcal M = \Omega\subset \mathbb{R}^n$, the  $A$-convexity condition also simplifies in that $\Omega$ is uniformly $A$-convex with respect to $G$ and $u$ if and only if
\begin{equation}
\kappa_1 > -ce^{-u} + h_{\partial\Omega}  \quad \mbox{on } \  \partial\Omega,
\end{equation}
 where $\kappa_1$ denotes the minimum curvature of $\partial\Omega$, and Theorem \ref{Th4.1} extends the second derivative estimates in \cite {JLL2007} for this special case with $c>0$. We remark though that the strictly regular case in Theorem \ref{Th4.1} also extends to general augmented Hessian equations and  corresponding second derivative estimates for \eqref {transformed Yamabe} for general $f$ are proved in \cite{JT2015}.

From Theorem \ref{Th4.1}, we can obtain existence theorems, which also extend Theorem \ref{Th1.2} and Corollary \ref{Cor 3.1} in the strictly regular case. First we
prove an appropriate extension of the gradient bound Lemma \ref{gradient bound}.
\begin{Lemma}\label{general gradient bound}
Let $u\in C^2(\bar\Omega)$ satisfy \eqref{weak convexity}
in a $C^2$ domain  $\Omega\subset\mathbb{R}^n$ and
\begin{equation} \label{oblique control}
|D_\beta u| \le \sigma_0,\quad \beta\cdot\nu \ge\beta_0
\end{equation}
on $\partial\Omega$, where $\beta \in L^\infty(\partial\Omega)$, $|\beta| = 1$
and  $\sigma_0$ and $\beta_0$  are positive constants. Then we have the estimate
\begin{equation}\label{gradient bound estimate}
|Du|\le C,
\end{equation}
where $C$ depends on $\mu_0, \sigma_0, \beta_0,  \Omega$ and $\sup |u|$.
\end{Lemma}

\begin{proof}
Invoking the tangential gradient $\delta u$, we have the formula
\begin{equation}\label{normal}
D_\nu u = \frac {1}{ \beta\cdot\nu}(D_\beta u - \beta\cdot\delta u)
\end{equation}
so that we can estimate
\begin{equation}
|Du| \le  \frac {1}{ \beta_0}(|\delta u| + \sigma_0) + |\delta u|
\end{equation}
on $\partial\Omega$, whence from \eqref {weak convexity}, we obtain
\begin{equation}\label {boundary convexity}
D^2 u \ge - \mu_1(1 + |\delta u|^2) I
\end{equation}
on $\partial\Omega$, for a further constant $\mu_1$, depending on $\mu_0$, $\beta_0$ and $\sigma_0$.
Now we consider in place of \eqref{gradient auxi}, the function
\begin{equation}\label{4.16}
w= e^{\kappa u}|\delta u|^2,
\end{equation}
so that at a point $x_0 \in \partial\Omega$ where $w$ is maximised we have
\begin{equation}
\begin{array}{rl}
0 \!\!&\!\! =\delta u\cdot \delta w \\
   \!\!&\!\! = e^{\kappa u}( \kappa|\delta u|^4 + 2 \delta_i u \delta_j u \delta_i \delta_j u) \\
   \!\!&\!\! = e^{\kappa u}[ \kappa|\delta u|^4 + 2 \delta_i u \delta_j u( D_{ij}u - D_{\nu} u \delta_i \nu_j)]\\
   \!\!&\!\! \ge e^{\kappa u}[ \kappa|\delta u|^4 -2\mu_1|\delta u|^2(1 + |\delta u|^2) - C |\delta u|^2],
\end{array}
\end{equation}
from \eqref {normal} and \eqref{boundary convexity}, where $C$ is a constant depending on $\beta_0$, $\sigma_0$
and $\partial\Omega$. By choosing $\kappa$ sufficiently large we conclude
the estimate \eqref{gradient bound estimate} on $\partial\Omega$ and the estimate in all of $\Omega$ then follows from
\cite{JTY2013} or Lemma \ref{gradient bound}.
\end{proof}

Lemma \ref{general gradient bound} provides an extension of Theorem 2.2 in \cite{LTU1986} to the weaker convexity condition \eqref{weak convexity}.
If we assume a stronger quadratic control from below on the Hessian, namely
\begin{equation}\label {stronger quadratic}
D_{ij} u\xi_i \xi_j \ge - \mu_0(1 + |D_\xi u|^2)
\end{equation}
for some constant $\mu_0$ and any unit vector $\xi$, we can reduce to Theorem 2.2 and the corresponding remark in \cite{LTU1986} as condition
\eqref{stronger quadratic} implies that the function $e^{\kappa u}$ is semi-convex for large $\kappa$. We also remark that the gradient estimates in Lemma \ref{gradient bound} and Lemma \ref{general gradient bound} have local versions.
In particular, if we fix any ball $B = B_R(x_0)$ of radius $R$ and centre $x_0\in\bar\Omega$, and suppose
$u\in C^2(\Omega\cap B)\cap C^1(\bar\Omega\cap B)$ satisfies \eqref{weak convexity} in $\Omega\cap B$ and
\eqref{oblique control} in $\partial\Omega\cap B$, then we have an estimate
\begin{equation}\label{local gradient estimate}
|Du(x_0)|\le \frac{C}{R},
\end{equation}
where $C$  depends on $\mu_0, \sigma_0, \beta_0,  \Omega$ and $\sup |u|$. To prove \eqref{local gradient estimate} we modify our proof of the global estimate Lemma \ref{general gradient bound} by maximizing in place of the
auxiliary functions
in \cite{JTY2013} and \eqref{4.16} above, the  functions
\begin{equation}
w_1 = \eta^2 e^{\kappa u}|D u|^2, \quad w_2 = \eta^2e^{\kappa u}|\delta u|^2
\end{equation}
over $\Omega\cap B$, $\partial\Omega\cap B$ respectively, where $\eta \in C^1_0(B)$ is a cut-off function chosen so that $0 \le\eta\le1$, $\eta(x_0) = 1$ and $|D\eta|\le 2/R$.

Note that \eqref{stronger quadratic} is satisfied in the special case \eqref{Euclidean Yamabe} so we obtain, for solutions of \eqref{transformed Yamabe},
\eqref {Euclidean Yamabe}, both local and global, gradient and second derivative estimates in terms of $\Omega$, $h_{\partial\Omega}$ and $\sup|u|$.

In order to apply Lemma \ref{general gradient bound}, we also need to assume that $G$ is uniformly oblique in the sense that
\begin{equation}\label{G uniformly oblique}
G_p(x,z,p)\cdot\nu \ge \beta_ 0,\quad |G_p(x,z,p)| \le \sigma_0 \quad \mbox{on } \  \partial\Omega,
\end{equation}
for all $x\in \Omega$, $|z|\le M_0$, $p\in \mathbb{R}^n$ and  positive constants $\beta_0$ and $\sigma_0$, depending on the constant $M_0$. Using the mean value theorem, we can thus write $G$ in the
semilinear form \eqref{G in oblique form} so that Lemma \ref{general gradient bound}, as well as the solution estimates
in Section \ref{Section 3}, are applicable.

We then have the following analogue of Theorem \ref{Th1.2} with a much weaker supersolution condition.

\begin{Theorem}\label{Th4.2}
Suppose that $A, B,G$ and $\Omega$ satisfy the hypotheses of Theorem \ref{Th4.1}  with $G$ uniformly oblique satisfying \eqref {G uniformly oblique}
and concave in $p$ for all
$(x,z,p)\in \partial\Omega\times\mathbb{R}\times\mathbb{R}^n$.  Assume also that $A$ and $B$
 are non-decreasing in $z$, $G$ is strictly decreasing in $z$,
A satisfies \eqref{QS} and that there exists a supersolution $\bar u$ and an elliptic subsolution $\underline u$ of equation \eqref{1.1} in
$C^2(\Omega)\cap C^1(\bar \Omega)$ satisfying $\mathcal G[\bar u] \le 0$ and $\mathcal G[\underline u] \ge 0$ respectively on $\partial\Omega$
with $\Omega$ uniformly $A$-convex with respect to $G$ and $\mathcal I = [ \underline u, \bar u]$.
Then the boundary value problem
\eqref{1.1}-\eqref{1.8} has a unique elliptic solution $u\in
C^{3,\alpha}(\bar \Omega)$ for any $\alpha<1$.
\end{Theorem}

Analogously to Corollary \ref{Cor 3.1}, we also have from Lemma \ref{Lemma 3.3} an existence theorem in the optimal transportation case.
Here we may also extend the condition \eqref{uniform monotonicity} by assuming there exists a positive constant $\gamma_0$
such that
\begin{equation}\label {G uniformly monotone}
G_z(x,z,p)\le-\gamma_0
\end{equation}
for all $(x,z,p)\in \partial\Omega\times\mathbb{R}\times\mathbb{R}^n$.

\begin{Corollary}\label{Cor 4.1}
 Suppose that equation \eqref{1.1} is a prescribed Jacobian equation of the form \eqref{PJE} generated by a cost function $c\in C^2(\mathcal{D})$ satisfying conditions A1 and A2 and $\mathcal{U}_x=\mathbb{R}^n$ for all $x\in \Omega$, with $\psi$ satisfying the structure conditions \eqref{psi OT case}, \eqref{sharp condition OT case}. Suppose also that $A, B, G$ and $\Omega$ satisfy the hypotheses of Theorem \ref{Th4.1}  with $G$ uniformly oblique satisfying \eqref {G uniformly oblique}, uniformly monotone satisfying \eqref{G uniformly monotone}
and concave in $p$ for all $(x,z,p)\in \partial\Omega\times\mathbb{R}\times\mathbb{R}^n$, $A$ satisfying \eqref{QS}, $B$ non-decreasing  and $\Omega$ uniformly $A$-convex with respect to $G$ and $-C$, where C is the constant in Lemma \ref{Lemma 3.3}. Then the boundary value problem \eqref{1.1}-\eqref{1.8} has a unique elliptic solution $u\in C^{3,\alpha}(\bar \Omega)$ for any $\alpha <1$.
\end{Corollary}

Finally we remark that when $G$ is assumed uniformly concave with respect to $p$, we only need $A$ to be regular in Theorems \ref{Th4.1}, \ref{Th4.2} and Corollary \ref{Cor 4.1} and the global second derivative estimates follow exactly as in Section 4 of \cite{TruWang2009}; see also \cite{Urbas1998}. Also the proof of Theorem \ref{Th4.1} would carry over to the cases when $G$ is non-increasing and $A$ is non-decreasing,  with either $G_z$ sufficiently small or $D_zA$ sufficiently large and $A$ again only assumed regular, (using in the first case the existence of an elliptic function and Lemma \ref{Lemma barrier}).


\end{document}